\documentclass[11pt]{article}
\usepackage[T2A]{fontenc}
\usepackage{amsmath}
\usepackage{amssymb}
\usepackage{amsthm}
\usepackage{arydshln}
\usepackage[colorlinks=true]{hyperref}
\usepackage{graphicx} 
\usepackage{mathabx}
\def\a{\mathbf{a}}
\def\b{\mathbf{b}}
\def\f{\mathbf{f}}
\def\m{\mathbf{m}}
\def\n{\mathbf{n}}
\def\N{\mathbb{N}}

\def\nuu{\boldsymbol{\nu}}
\def\etta{\boldsymbol{\eta}}
\def\omm{\boldsymbol{\omega}}
\def\zetta{\boldsymbol{\zeta}}
\newtheorem{theorem}{\hspace*{\parindent}Theorem}
\newtheorem{proposition}{\hspace*{\parindent}Proposition}
\newtheorem*{theorem*}{\hspace*{\parindent}Theorem}
\newtheorem{lemma}{\hspace*{\parindent}Lemma}
\newtheorem{corollary}{\hspace*{\parindent}Corollary}
\newtheorem{open problem}{\hspace*{\parindent}{Open problem}}
\newtheorem{example}{\hspace*{\parindent}{Example}}

\DeclareMathOperator*{\sign}{sign}

\title{Generalized $f$-Eulerian polynomials: zeros and hypergeometric representations with applications}
\author{Dmitrii\:Karp$^{\rm a,b}$\footnote{Corresponding author.  E-mail: D. Karp -- \emph{dimkrp@gmail.com}; A.Vishnyakova -- \emph{annalyticity@gmail.com}}~~and Anna\:Vishnyakova$^{\rm a}$
	\\[10pt]
	\small{\textit{$\phantom{1}^a$Holon Institute of Technology, Holon, Israel}}
    \\[10pt]
	\small{\textit{$\phantom{1}^b$Institute of Mathematics and Informatics, BAS, Bulgaria}}
}
\date{}
\begin{document}
	\maketitle
	
	\begin{abstract}
		In this paper, we explore (slightly generalized) $f$-Eulerian polynomials introduced by Stanley and frequently appearing in combinatorics. Notable special cases include the classical Eulerian polynomials, the generating polynomials of order polynomials for certain labeled posets, and the $d$-Narayana polynomials. We establish simple sufficient conditions  for the reality (and sign) of their zeros and present implications for total positivity of sequences generated by values of polynomials at integers. We further relate these polynomials to generalized Euler's transformations for the generalized hypergeometric functions with integral parameter differences. Exploiting this and other hypergeometric connections, we provide purely hypergeometric proofs for various known and some new properties of $d$-Narayana polynomials. Another family encompassed by our definition of the generalized $f$-Eulerian polynomials is that of Jacobi-Pi\~neiro type II multiple orthogonal polynomials. Their zero location can thus be analyzed, for both canonical and non-canonical parameter values, without invoking orthogonality.  Finally, we present several connection formulas relating $d$-Narayana polynomials to particular Jacobi-Pi\~neiro polynomials.
	\end{abstract}
	
	\bigskip
	
	
	\bigskip
	
	
	\bigskip
	
	\section{Introduction and preliminaries}
	Suppose $f$ is a polynomial of degree $d$.  The function $\hat{w}$ defined by the series
	\begin{equation}\label{eq:f-Euler}
		\sum\limits_{n=0}^{\infty}f(n)x^n=\frac{\hat{w}(x)}{(1-x)^{d+1}}
	\end{equation}
	is known (and easily seen) to be a polynomial of degree at most $d$ \cite[Lemma~3.15.11]{StanleyBook1}, frequently called $f$-Eulerian polynomial \cite[section 4.3, p.\,473]{StanleyBook1}. The name is explained by the observation that for $f(n)=n^d$, we obtain $\hat{w}(x)=A_{d}(x)$, where $A_d(x)=\sum_{k=1}^{d}A(d,k)x^k$ is the Euler polynomial generated by the Euler numbers $A(d,k)$ counting the number of permutations of $\{1,2,\ldots,d\}$  with exactly $k-1$ descents, i.e. pairs $\sigma(i)>\sigma(i+1)$. If $f(n)=\binom{n+d-i}{d}$ (the multiset coefficient),  we obtain the monomials: $\hat{w}(x)=x^i$.  Another example of combinatorial significance is the generating polynomials of the order polynomials $\Omega_{P,\omega}(m)$ of  labeled posets, see \cite[Theorem~3.15.8]{StanleyBook1} and the text preceding it.

	One further example, we will discuss in a greater detail in Section~4, is $d$-Narayana polynomials defined by Robert Sulanke \cite{Sulanke2004}. They have several interpretations in combinatorics and, more recently, in the study of Grassmanians \cite{Braun2019}, explained in Section~4.  Sulanke's results imply that the $d$-Narayana polynomials $N_{d,m}$ satisfy 
\begin{equation}\label{eq:CiglerF}
		{}_dF_{d-1}\!\left(\!\!\begin{array}{c}m+1,m+2,\ldots,m+d\\2,3,\ldots,d\end{array}\vline\:x\right)=\frac{N_{d,m}(x)}{(1-x)^{md+1}},
\end{equation}
which is a particular case of \eqref{eq:f-Euler}.  Here and below ${}_{p}F_{q}$ stands for the generalized hypergeometric function \cite[Chapter~16]{NIST}. 
 Chen, Yang and Zhang proved in \cite[Theorem 3.1]{CYZ} that all zeros of  $N_{d,m}(x)$  are real (and negative) resolving and strengthening a previous conjecture by Kirillov claiming that the sequence of $d$-Narayana numbers (i.e. the coefficients of $N_{d,m}(x)$) is log-concave.  Their proof uses bijection with order polynomials of column strict Ferrers posets whose zeros have been previously proved to be real by Brenti \cite[proof of Theorem~5.3.2]{BrentiMemoir} settling a particular case of Neggers-Stanley conjecture.  For general $f$-Eulerian polynomial $\hat{w}(x)$ from \eqref{eq:f-Euler}, Brenti gave the following sufficient condition for their zeros to be real and non-positive \cite[Theorem~3.7]{BrentiContMath}, \cite[Theorem~4.4.1]{BrentiMemoir}: suppose $f$ is a polynomial with real coefficients having only real zeros and $f(x)=0$ for each integer $x$ in the set $[\lambda(f),-1]\cup[0,\Lambda(f)]$, where $\lambda(f)$ ($\Lambda(f)$) stands for the smallest (largest) zero of $f$, then  $\hat{w}(x)$ has  non-negative coefficients and  only  real zeros.  This theorem was applied by Agapito \cite{Agapito} to a subset of $f$-Eulerian polynomials $P_N(x)$ defined for given integers $1\le{r}\le{s}$, $n\ge1$, by the relation \cite[Definition~3.4]{Agapito}
	$$
	\bigg[x^r\Big(\frac{d}{dx}\Big)^s\bigg]^n\frac{1}{1-x}=\frac{K(r,s,n)x^{r-1}P_{N}(x)}{(1-x)^{ns+1}},
	$$
	where $N=r(n-1)+1$. The Euler polynomials $A_n(x)$ correspond to $r=s=1$, while $d$-Narayana polynomials appear by setting $s=d$, $r=d-1$, $m=n$ \cite[Remark~3.3]{Agapito}.  In  \cite[Theorem~3.13]{Agapito} Agapito applied Brenti's result to show that $P_N(x)$ has only real non-positive zeros, a fact apparently overlooked by the authors of \cite{CYZ}.  He further established positivity of the so-called $\gamma$ coefficients of $P_N$, see \cite[section~4.2]{Agapito}.  We will use the hypergeometric definition \eqref{eq:CiglerF} to establish further properties of $d$-Narayana polynomials, namely, a formula for the value $N_{d,m}(1)$ known as multidimensional Catalan number, the self-reciprocity, expansions in monomial and Bernstein bases.  Reality and negativity of zeros of $N_{d,m}(x)$ will also follow from our results. These investigations were motivated by a question that Johann Cigler asked the first author in a letter dated Fabruary 28, 2022. 
	
In this paper, we will examine a slightly more general class of polynomials than \eqref{eq:f-Euler} generated by 
\begin{equation}\label{eq:fa-Euler}
		F\!\left(\!\!\begin{array}{c}a\\-\end{array}\bigg\vert\,F_m\,\bigg\vert\,x\right):=\sum\limits_{n=0}^{\infty}F_m(n)\frac{(a)_n}{n!}x^n=\frac{\hat{w}(x)}{(1-x)^{m+a}},
\end{equation}
where $a$ can be any real number, $(a)_k=\Gamma(a+k)/\Gamma(a)$ is Pochhammer's symbol,  and $F_m(n)$ stands for a polynomial in $n$ of degree $m$ normalized by $F_m(0)=1$  (if $F_m(0)=0$ one can cancel the factor $n$ in $F_m$ with the same factor in $n!$ and shift the summation index). If $a$ is a positive integer, then \eqref{eq:fa-Euler} is easily seen to reduce to \eqref{eq:f-Euler}, as $(a)_n/n!$ becomes a polynomial in $n$.  For general real $a$, the form \eqref{eq:fa-Euler} is a true generalization of \eqref{eq:f-Euler}. We will frequently write
$$
\hat{w}(x)=\hat{w}\big(a;F_m\,\big\vert\,x\big)
$$
to emphasize the dependence of $\hat{w}$ on $a$ and $F_m$. The polynomial $F_m$ can always be written in the form 
\begin{equation}\label{eq:Fm-defined}
	F_m(t)=\frac{(f_1+t)_{m_1}(f_2+t)_{m_2}\cdots(f_r+t)_{m_r}}{(f_1)_{m_1}(f_2)_{m_2}\cdots(f_r)_{m_r}}, \text{ so that }F_m(n)=\frac{(f_1+m_1)_{n}\cdots(f_r+m_r)_{n}}{(f_1)_{n}\cdots(f_r)_{n}},
\end{equation}
where we applied $(f+n)_{m}/(f)_{m}=(f+m)_{n}/(f)_{n}$. The numbers  $-f_1,-f_1-1,\ldots,-f_1-m_1+1$, $\ldots$, $-f_r,-f_r-1,\ldots,-f_r-m_r+1$ are the zeros of $F_m$, which we will usually assume to be real in this paper. These zeros are partitioned into $r$ blocks with $m_i$ zeros within block $i$ separated by unity, so that the parameters $\m=(m_1,\ldots,m_r)$ are positive integers and $m=m_1+m_2+\cdots+m_r$. Any block may contain just one zero, if no other zero is at distance one from it.   This allows rewriting the \eqref{eq:fa-Euler} in terms of generalized hypergeometric functions as follows:
\begin{equation}\label{eq:fa-Euler1}
		F\!\left(\!\!\begin{array}{c}a\\-\end{array}\bigg\vert\,F_m\,\bigg\vert\,x\right)={_{r+1}F_{r}}\!\left(\!\!\begin{array}{c}a,\f+\m\\\f\end{array}\vline\,x\right)=\frac{\hat{w}\big(a;F_m\,\vert\,x\big)}{(1-x)^{m+a}}.
\end{equation}

There are several reasons for considering the form  \eqref{eq:fa-Euler}, \eqref{eq:fa-Euler1} instead of \eqref{eq:f-Euler}.  One is its relation to the generalized Euler transformations of hypergeometric functions \cite{Karp2025,MP2013}.  Another comes from our approach to establishing conditions for reality of zeros of $\hat{w}(x)$ allowing for generalization of Brenti's theorem about negative zeros to any $a>0$ and leading to new sufficient conditions ensuring that all zeros of $\hat{w}$ lie in $(0,1)$ or outside of $(0,1)$. Our proofs only use elementary analysis and are entirely different from Brenti's.  

Yet another reason stems from the observation that by using \eqref{eq:fa-Euler} or \eqref{eq:fa-Euler1} we get a unified treatment of $f$-Eulerian (in particular, $d$-Narayana) and Jacobi–Pi\~neiro multiple orthogonal polynomials. Type II Jacobi–Pi\~neiro polynomials are orthogonal with respect to the system of measures $\left(\mu_1, \ldots, \mu_r\right)$ with $d\mu_i(x)=x^{\alpha_i}(1-x)^\beta dx$ on $[0,1]$, where  $\beta>-1$ and $\alpha_1, \ldots, \alpha_r>-1$ satisfy $\alpha_i-\alpha_j \notin \mathbb{Z}$ whenever $i \neq j$. Explicit formula was discovered in \cite[Theorem~3.2]{BCvA} and is given by \cite[(23.3.5)]{IsmailBook}
	\begin{equation}\label{eq:JacobiPineiro}
		P_{\n}^{(\boldsymbol{\alpha}, \beta)}(x)=\frac{(-1)^{n}\left(\boldsymbol{\alpha}+1\right)_{\n}}{ \left(n+\boldsymbol{\alpha}+\beta+1\right)_{\n}(1-x)^\beta}\: {}_{r+1}F_r\left(\left.\!\!\!\begin{array}{c}
			-n-\beta, \boldsymbol{\alpha}+\n+1 \\
			\boldsymbol{\alpha}+1
		\end{array} \right\rvert x\right),	
	\end{equation}
	where $\n=(n_1,\ldots,n_r)$, $n=n_1+\cdots+n_r$ and the shorthand notation $(\boldsymbol{\alpha})_{\n}=(\alpha_1)_{n_1}\cdots(\alpha_r)_{n_r}$ has been used. Formula \eqref{eq:JacobiPineiro} is equivalent to \eqref{eq:fa-Euler1} with $a<0$.  
   According to \cite[Theorem 23.1.4]{IsmailBook} all $n$ zeros of $ P_{\n}^{(\boldsymbol{\alpha}, \beta)}(x) $ lie in $(0,1)$.
Note that the connection between the standard Narayana polynomials $N_{2,m}$ and the standard Jacobi polynomials $P^{\alpha,\beta}_n$ was found in \cite[Proposition~6]{KMFS}.  This connection will be generalized in subsection~5.3 to the case of $d$-Narayana and Jacobi–Pi\~neiro polynomials.

This paper is organized as follows.   In Section~2 we study zeros of the generalized $f$-Eulerian polynomials \eqref{eq:fa-Euler}. We present sufficient conditions for zeros to be negative, lie in $(0,1)$ or lie outside of $(0,1)$. We further link the zeros of $\hat{w}$ to total positivity of the generating sequence $F_m(n)(a)_n/n!$ from \eqref{eq:fa-Euler} and discuss some related open problems. In Section~3 we establish a connection between the generalized $f$-Eulerian polynomials \eqref{eq:fa-Euler} and the generalized Euler transformations for hypergeometric functions known as Miller-Paris transformations.  This connection yields new expansions in monomial and Bernstein bases.  Section~4 is devoted to investigation of $d$-Narayana polynomials using both the results from Sections 2 and 3 and other hypergeometric tools. In this section, we answer in the affirmative to the questions asked in the letter by Johann Cigler mentioned earlier.  Final Section~5 explores Jacobi-Pi\~neiro polynomials by using their hypergeometric form and without resorting to orthogonality. 
We show that the location of their zeros can be derived directly from the hypergeometric generating relation \eqref{eq:JacobiPineiro} using our results, which also provides information about the zeros for non-canonical parameter values.

\section{Zeros of \texorpdfstring{$f$}{f}-Eulerian polynomials}	

\subsection{Negative zeros and totally positive sequences}
	
We will repeatedly use the shorthand notation 
\begin{equation*}
(\f)_{\m}=(f_1)_{m_1}\cdots(f_r)_{m_r}~~\text{and}~~\f+\beta=(f_1+\beta,f_2+\beta,\ldots,f_r+\beta).
\end{equation*}
 As already noticed in the introduction we can factor the polynomial $F_{m}$ as follows: $(\f)_{\m}F_{m}(k)=(k+\alpha_1)\cdots(k+\alpha_m)$, where $\alpha_j$ are the negated roots of $F_m$, i.e. the numbers 
	$f_1,f_1+1,\ldots, f_1+m_1-1$, $f_2,f_2+1,\ldots,f_2+m_2-1$, $\ldots$, $f_r, f_r+1,\ldots, f_{r}+m_{r}-1$.  Assume, without loss of generality, that they are ordered ascending, i.e. 
	$\alpha_{1}\le \alpha_{2}\le \cdots\le \alpha_{m}$.
	We can write
	\begin{equation}\label{eq:Tproduct}
		(\f)_{\m}\cdot{}_{r+1}F_{r}\left.\!\!\left(\!\begin{matrix}a, \f+\m\\\f\end{matrix}\right\vert x\right)=\sum\limits_{k=0}^{\infty}(\f+k)_{\m}\frac{(a)_kx^k}{k!}=M_{\alpha_m}M_{\alpha_{m-1}}\cdots M_{\alpha_1}(1-x)^{-a},	
	\end{equation}
	where the operator $M_{\alpha}$, $\alpha\in\mathbb{R}$, is defined by its action on a formal power series as follows: 
	\begin{equation}\label{eq:Talpha-defined}
		M_{\alpha}\left(\sum\limits_{k=0}^{\infty}a_kx^k\right)=
		\sum\limits_{k=0}^{\infty}(k+\alpha)a_kx^k.
	\end{equation}
	If the series has positive radius of convergence and $f(x)=\sum_{k=0}^{\infty}a_kx^k$, we clearly have
	$$
	M_{\alpha}f(x)=x^{1-\alpha}Dx^{\alpha}f(x),
	$$
	where $D$ denotes the derivative with respect to $x$. This formula is also valid for divergent series if understood term-wise. For series of the form \eqref{eq:fa-Euler} we are interested in how the action of $M_{\alpha}$ affects the reality of zeros of the $f$-Eulerian polynomial $\hat{w}$.  We give a partial answer to this inquiry in
    Lemmas~\ref{lm:Talpha-positive_s}, \ref{lm:Talpha-negative_s} and \ref{lm:Talpha-zeros-with-gap}.  Iterative application of these lemmas will lead to certain statements about zeros of $\hat{w}$ presented in 
    Theorems~\ref{th:negative_zeros}, \ref{th:zeros-in-0-1} and  \ref{th:zeros_greater_1} below. The first theorem hinges on the following 
	\begin{lemma}\label{lm:Talpha-positive_s}
		Suppose $s>0$ and
		\begin{equation}\label{eq:f-Pn}
			f(x)=\frac{P_n(x)}{(1-x)^s},
		\end{equation}
		where $P_n$ is a polynomial of degree $n\ge0$ with only real negative zeros \emph{(}a real constant if $n=0$\emph{)}. Then for any $\alpha\in\mathbb{R}$ 
		the function $M_{\alpha}f$ has the form
		\begin{equation}\label{eq:Talpha-f}
			M_{\alpha}f(x)=\frac{\hat{P}(x)}{(1-x)^{s+1}},
		\end{equation}
		where the polynomial $\hat{P}(x)$ has degree $n+1$ if $\alpha\ne s-n$ or degree $n$ if $\alpha=s-n$. All zeros of $\hat{P}$ are real if $\alpha>0$ or if $\alpha<0$ and $s-\alpha-n>0$.  
		
		More precisely, if $\alpha>0$, $s-\alpha-n>0$, then all zeros of $\hat{P}(x)$ are negative and 
		\begin{equation}\label{eq:interlacing_allnegative}
			\hat{x}_{n+1}\le x_{n}\le\hat{x}_{n}\le\cdots\le x_1\le\hat{x}_{1}<0,
		\end{equation}
		where $x_j$ \emph{(}$\hat{x}_j$\emph{)} are zeros of $P_n$ \emph{(}$\hat{P}$\emph{)}; if $s-\alpha-n=0$ the above relation remains valid with $\hat{x}_{n+1}$ omitted;  if $\alpha>0$ and $s-\alpha-n<0$ all but one zero of $\hat{P}(x)$  are negative and
		\begin{equation}\label{eq:interlacing_onepositive}
			x_{n}\le\hat{x}_{n+1}\le\cdots\le x_1\le\hat{x}_2<0<1<\hat{x}_{1}.
		\end{equation}
		If $\alpha<0$ and $s-\alpha-n>0$, then all but one zero of $\hat{P}(x)$  are negative and
		\begin{equation}\label{eq:interlacing_negative_alpha}
			\hat{x}_{n+1}\le x_{n}\le\hat{x}_{n}\le\cdots\le x_{2}\le\hat{x}_2<x_{1}<0<\hat{x}_{1}<1.
		\end{equation}
		Moreover, if all zeros of $P_n(x)$ are simple, then so are the zeros of $\hat{P}(x)$ and all inequalities in \eqref{eq:interlacing_allnegative}, \eqref{eq:interlacing_onepositive} and \eqref{eq:interlacing_negative_alpha} are strict.
	\end{lemma}
	
	\begin{proof} As all zeros of $P_n$ are negative, all its coefficients have the same sign, which we can, without loss of generality, assume to be positive.  Next, suppose that  all zeros of $P_n(x)$ are simple and assume for a moment that $x>0$. Differentiation yields
		\begin{multline*}
			M_{\alpha}f(x)=x^{1-\alpha}D\frac{x^{\alpha}P_n(x)}{(1-x)^s}
			=\frac{(\alpha x^{\alpha-1}P_n(x)+x^{\alpha}P_n'(x))(1-x)^s+sx^{\alpha}P_n(1-x)^{s-1}}{x^{\alpha-1}(1-x)^{2s}}
			\\
			=\frac{x(1-x)P_n'(x)+(x(s-\alpha)+\alpha)P_n(x)}{(1-x)^{s+1}}=\frac{\hat{P}(x)}{(1-x)^{s+1}}.
		\end{multline*}
		The expression for $\hat{P}(x)$ from the ultimate equality is obviously true irrespective of the sign of $x$. The numerator polynomial 
		\begin{equation}\label{eq:hatP}
			\hat{P}(x)=x(1-x)P_n'(x)+(x(s-\alpha)+\alpha)P_n(x)	
		\end{equation}
		has the leading term $a_{n}(s-\alpha-n)x^{n+1}$, where $P_n(x)=a_nx^n+a_{n-1}x^{n-1}+\cdots$, and, hence, for $s-\alpha-n\ne0$ the degree of  $\hat{P}$ is indeed $n+1$. If $s-\alpha-n=0$, the coefficient at $x^n$ becomes $sa_{n}+a_{n-1}\ne0$ since both $a_{n}$ and $a_{n-1}$ do not vanish and have the same sign, while $s>0$. This proves that the degree of  $\hat{P}$ is exactly $n$ if $s-\alpha-n=0$.
		
		First, we will treat the case $\alpha>0$. Taking $x=0$ in \eqref{eq:hatP} we see that
		$\hat{P}(0)=\alpha P_{n}(0)>0$.  Denoting by $x_1$ the largest zero of $P_{n}$ we obtain  $\hat{P}(x_1)=x_1(1-x_1)P_n'(x_1)$. As $P_{n}(x)>0$ for all $x>x_1$ and $x_1$ is a simple zero, we conclude that $P_n'(x_1)>0$ and hence $\hat{P}(x_1)<0$.  Considering similarly the values $\hat{P}(x_j)$ at further zeros  $x_2>x_3>\cdots>x_{n}$ of $P_n(x)$, we conclude that $(-1)^{j}\hat{P}(x_j)>0$ for $j=1,\ldots,n$. Hence, there is an odd number of zeros of $\hat{P}$ between any two consecutive zeros of $P_n$ plus an odd number of zeros of $\hat{P}$ between $x_1$ and $0$.  If any of these odd numbers were $>1$, $\hat{P}$ would have had at least $n+2$ zeros which is impossible as its degree does not exceed $n+1$. Hence, there is exactly one zero of $\hat{P}$ between any two consecutive zeros of $P_n$ plus one zero of $\hat{P}$ between $x_1$ and $0$. If degree of $\hat{P}$ is $n$, i.e. the coefficient $s-\alpha-n$ of the leading term vanishes, then these are all the zeros of $\hat{P}$ and we have established \eqref{eq:interlacing_allnegative} with strict inequality signs and with no $\hat{x}_{n+1}$.  If $s-\alpha-n>0$, then  $\hat{P}(x)>0$ for all $x\ge0$, so that one remaining zero of $\hat{P}$ lies in the interval $(-\infty,x_{n})$ which proves \eqref{eq:interlacing_allnegative} with strict signs. If $s-\alpha-n<0$,  then $\hat{P}(x)<0$ for sufficiently large positive $x$, so that the remaining zero must be positive, but as both $\hat{P}(0)=\alpha P_{n}(0)>0$ and $\hat{P}(1)=sP_n(1)>0$, it must be greater than $1$ proving \eqref{eq:interlacing_onepositive}.

		Next, suppose $\alpha<0$, $s-\alpha-n>0$.  In this case 	$\hat{P}(0)=\alpha P_{n}(0)<0$.  As the leading coefficient $s-\alpha-n>0$ there is an odd number of zeros of $\hat{P}$ on $(0,\infty)$. At the same time  $\hat{P}(x_j)=x_{j}(1-x_{j})P_{n}'(x_j)$ and we conclude, as before,  that $(-1)^{j}\hat{P}(x_j)>0$ for $j=1,2,\ldots,n$. This implies that there are $n-1$ zero of $\hat{P}$ lying in $n-1$ intervals with endpoints $x_j$, $j=1,\ldots,n$.  Hence, there is exactly one positive zero (as there must be an odd number of them) and exactly one zero on the left of $x_n$.  At the same time as 	$\hat{P}(1)=sP_n(1)>0$ the positive zero lies in $(0,1)$. This proves \eqref{eq:interlacing_negative_alpha} with strict inequalities.   
		
		Finally note that if zeros of $P_n(x)$ are not simple, there exists an arbitrarily small perturbation of $P_n(x)$ with only simple negative zeros, so that by the above argument all zeros of the perturbed $\hat{P}$ satisfy the claimed properties.   Formula \eqref{eq:hatP} shows that the perturbed $\hat{P}$ becomes  $\hat{P}$  as perturbation of  $P_n(x)$ tends to zero. 
	\end{proof}

		Recursive application of the Lemma~\ref{lm:Talpha-positive_s} to decomposition \eqref{eq:Tproduct} yields our first theorem. 
	\begin{theorem}\label{th:negative_zeros}
		Denote by $(\alpha_{1},\ldots,\alpha_{m})$ the zeros of the polynomial $F_m(-t)$ in \eqref{eq:fa-Euler} \emph{(}multiple zeros are repeated according to multiplicity\emph{)} and assume that $a>0$. If 
		\begin{equation}\label{eq:condbrenti0}
			\begin{split}
				&0 \leq \alpha_{1} \leq \alpha_{2} \leq \cdots \leq \alpha_{p}~~\text{and}
				\\
				&\big(\{a,a+1,a+2,\ldots\}\cap[a,\alpha_{p})\big) \subset \{ \alpha_{1}, \alpha_{2}, \cdots, \alpha_{p}\}
			\end{split}
		\end{equation}
		for $p=m$, then all zeros of the $F_m$-Eulerian polynomial $\hat{w}(a;F_m\vert\:x)$ defined in \eqref{eq:fa-Euler} or, equivalently, of  the hypergeometric function on the left hand side of \eqref{eq:Tproduct} are real and negative.  If \eqref{eq:condbrenti0} holds for $p=m-1$ and $\alpha_{m}$ is any real number, then all zeros of $\hat{w}(a;F_m\vert\:x)$ are real and at most one of them is positive. 
	\end{theorem}
	\textbf{Remark.} The above theorem coincides  with Brenti's result \cite[Theorem~3.7]{BrentiContMath}, \cite[Theorem~4.4.1]{BrentiMemoir} mentioned in the introduction when $a=1$ and \eqref{eq:condbrenti0} holds for $p=m$.  Indeed, in this case Brenti's condition requires that all integers $1,2,\ldots\lfloor \alpha_m\rfloor$ are contained in the set $(\alpha_{1},\ldots,\alpha_{m})$ which is the same as \eqref{eq:condbrenti0} .   For the case when some of the roots of $F_m(t)=0$ are positive Brenti's theorem gives information not contained in Theorem~\ref{th:negative_zeros}.      
\begin{proof}
Denote $s_k=a+k$, $k=0,1,\ldots$.  Assuming \eqref{eq:condbrenti0}, according to Lemma~\ref{lm:Talpha-positive_s} we get
		$$
		M_{\alpha_1}\frac{1}{(1-x)^{s_0}}=\frac{P_{n_1}(x)}{(1-x)^{s_1}}
		$$
		where $n_1=\mathrm{deg}(P_{n_1})=1$ if $\alpha_1<s_0$ or $n_1=0$ if $\alpha_1=s_0$. Proceeding similarly, after $p$ steps we get
		$$
		M_{\alpha_p}{\cdots}M_{\alpha_2}M_{\alpha_1}\frac{1}{(1-x)^{s_0}}=\frac{P_{n_p}(x)}{(1-x)^{s_p}},
		$$
		where $n_p=p-|\{\alpha_1,\ldots,\alpha_p\}\cap\{a,a+1,\ldots,a+p\}|$.  Hence, if  \eqref{eq:condbrenti0} holds for $p=m$ the claim follows by \eqref{eq:interlacing_allnegative} and the text around it.  If \eqref{eq:condbrenti0} holds for $p=m-1$ and  
		$$
		\alpha_m>s_{p}-n_{p}=a+|\{\alpha_1,\ldots,\alpha_p\}\cap\{a,a+1,\ldots,a+p\}|
		$$
		the claim follows by \eqref{eq:interlacing_onepositive}. If $\alpha_m<0$ the condition 
		$s_p-n_p>\alpha_{m}$ is automatically true and the claim follows by \eqref{eq:interlacing_negative_alpha}.
	\end{proof}

\textbf{Remark}. Applications of Lemma~\ref{lm:Talpha-positive_s} are not limited to generalized $f$-Eulerian polynomials \eqref{eq:fa-Euler} as we will demonstrate now.  In his recent preprint \cite{He2023} Tian-Xiao He introduced $m$-Eulerian numbers and their generating polynomials denoted $S_{m;n}(x)$.  These numbers count the permutations of the multiset $M_n=\{1,\ldots,1,2,\ldots,2,\ldots, n,\ldots,n\}$, where each integer is repeated $m$ times, with exactly $k$ descents, i.e. such that $a_{j}>a_{j+1}$ or $j=mn$ for exactly $k$ values of $j\in\{1,2,\ldots,mn\}$. Here $a_1,\ldots,a_{mn}$ are permutations of $M_n$ satisfying the following property: if $u<v<w$ and $a_{u}=a_{w}$, then $a_{u}\ge a_{v}$ \cite[Definition~2.1]{He2023}.  According to \cite[(11)]{He2023} these polynomials satisfy:
$$
S_{m;n+1}(x)=x(1-x)S_{m;n}'(x)+(1+mnx)S_{m;n}(x)
$$
This formula can be identified with \eqref{eq:hatP} choosing $s=mn+1$, $\alpha=1$.  As $S_{m;0}=1$ we can apply Lemma~\ref{lm:Talpha-positive_s}
recursively to conclude that all zeros of $S_{m;n}(x)$ are real and negative for all $m,n\ge1$.  In particular, this proves that  the sequence of $m$-Eulerian numbers is log-concave, and, moreover, totally positive.  Note that polynomials $S_{m;n}(x)$ for $m\ge2$ are not of the type \eqref{eq:fa-Euler}  as the generating function \cite[(15)]{He2023} is not hypergeometric.

\medskip

\begin{corollary}\label{cr:large-a}
		Suppose $f_{j}+m_{j}<a$ for all $j=1,\ldots,r$. Then all zeros of the  $F_m$-Eulerian 
		polynomial $\hat{w}(a;F_m\vert\:x)$, and the hypergeometric function on the left hand 
		side of \eqref{eq:fa-Euler1} are real and negative. 
\end{corollary}
Theorem~\ref{th:negative_zeros} only provides sufficient conditions on $F_m$ in order that $\hat{w}$ had only real negative zeros, which are in no way necessary as we will see shortly in Examples~\ref{ex1} and \ref{ex2}.  Hence, to the best of our knowledge the following problem remains open even for $a=1$:
\begin{open problem}\label{pr1}
Given $a>0$ describe the set of polynomials $F_m\in\mathbb{R}[x]$ such that $\hat{w}$ generated by
\eqref{eq:fa-Euler} has only real negative zeros. This could be asked either for a given fixed $m\in\N$ or for all $m\in\N$. 
\end{open problem}
When $a\in\N$ the above problem can be restated as a problem about total positivity of the sequences of the form $\{F_m(n)(a)_n/n!\}_{n\ge0}$.  Recall that real sequence $(a_k)_{k=0}^\infty$  is called totally positive ($TP$) or P\'{o}lya frequency if all minors of the infinite matrix
\begin {equation}\label{mat}
		\left\|
		\begin{array}{ccccc}
			a_0 & a_1 & a_2 & a_3 &\ldots \\
			0   & a_0 & a_1 & a_2 &\ldots \\
			0   &  0  & a_0 & a_1 &\ldots \\
			0   &  0  &  0  & a_0 &\ldots \\
			\vdots&\vdots&\vdots&\vdots&\ddots
		\end{array}
		\right\|
\end {equation}
are nonnegative.  The following celebrated theorem due to  
Aissen,	Schoenberg, Whitney and Edrei \cite{aissen,Edrei} (see also \cite[p. 412]{tp}) characterizes $TP$ sequences $\{a_n\}_{n\ge0}$ in terms of their generating functions $f(z)=\sum_{n\ge0}a_nz^n$.

\begin{theorem*}[ASWE] A function $f$ is a generating function of a $TP$ sequence if and only if
\begin{equation}\label{eq:aswe}
f(z)=C z^q e^{\gamma z}\prod_{k=1}^\infty \frac{ (1+\alpha_kz)}{(1-\beta_kz)},
\end{equation}
where $C\ge 0$, $q\in\mathbb{N}\cup\{0\}$, $\gamma\ge0$, $\alpha_k,\beta_k\ge 0$, $\sum_{k=1}^\infty(\alpha_k+\beta_k)
<\infty$. 
\end{theorem*}

Suppose $F_m$ is a polynomial with positive coefficients and $a\in\N$.  By theorem ASWE, the sequence 
$\left(F_m(n)\frac{(a)_n}{n!}\right)_{n=0}^\infty$ is totally positive if and only if the 
corresponding polynomial $\hat{w}$ has  only  real negative zeros. Hence, Theorem~\ref{th:negative_zeros} can be restated as follows:
\begin{theorem}\label{th:total-pos}
		Let $a \in \mathbb{N}$ and  $0 \leq \alpha_{1} \leq \alpha_{2} \leq \cdots \leq \alpha_{m}$ 
		be given real numbers such that  
		\begin{equation}
			\label{condbrenti}
			\left( [a,   \alpha_{m}) \cap \mathbb{N}\right) \subset \{ \alpha_{1}, \alpha_{2},  
			\cdots ,  \alpha_{m} \}. 
\end{equation}
Then the sequence  
\begin{equation}\label{totpos}\left(\frac{(a)_k}{k!}(k+\alpha_1)(k+\alpha_2)\cdot 
			\ldots \cdot (k+\alpha_m)\right)_{k=0}^\infty
		\end{equation} 
		is totally positive.
	\end{theorem}
The following examples show that the conditions of the above theorem are not necessary.
\begin{example}\label{ex1}
Set $F_2(z) =(z+\frac{11}{10})^2$. Then conditions of Theorem~\ref{th:total-pos} are clearly violated. Nevertheless,
\begin{equation}\label{exxx1} \sum_{n=0}^\infty \Big(n+\frac{11}{10}\Big)^2 
x^n= \frac{x^2+78x +121}{100(1-x)^3},
\end{equation} 
with $\hat{W}(x)=x^2+78x +121$ having two negative real zeros, so that the sequence $\left\{(n+\frac{11}{10})^2\right\}_{n=0}^\infty$ 
is totally positive.
\end{example}

\begin{example} \label{ex2} 
Set $F_2(x)=z^2+\frac{1}{8}$. Its zeros are non-real, but  
\begin{equation}\label{exxx2}
\sum_{n=0}^\infty \Big(n^2 +\frac{1}{8}\Big) 
x^n= \frac{(3x+1)^2}{8(1-x)^3},
\end{equation}
so that $\hat{w}(x)=(3x+1)^2$ has one real negative zero of multiplicity $2$. 
Hence, the sequence $\left\{(n^2+\frac{1}{8})\right\}_{n=0}^\infty$ is totally positive. 
\end{example}

Generalizing two previous examples, a simple calculation yields
$$
\sum _{k=0}^{\infty}(k^2+bk+c)\frac{(a)_k}{k!}x^k=\frac{(a^2-ab+c)x^2+(a(b+1)-2c)x+c}{(1-x)^{a+2}}.
$$
This implies that, under the assumption $a>0$, the numerator polynomial has negative real roots if and only if  $a(b+1)^2-4c(a+1)\ge0$, $c>0$, $a(b+1)>2c$ and $a^2-ab+c>0$.  These conditions can be satisfied even when  the roots of $k\to k^2+bk+c$  are complex, i.e. $b^2<4c$.  But even assuming the roots  $-\alpha,-\beta$ are real, we get:
$$
\sum _{k=0}^{\infty}(k+\alpha)(k+\beta)\frac{(a)_k}{k!}x^k=\frac{(\alpha\beta-a(\alpha+\beta)+a^2)x^2+(a(\alpha+\beta+1)-2\alpha\beta)x+\alpha\beta}{(1-x)^{a+2}}=\frac{\hat{w}_{a,\alpha,\beta}(x)}{(1-x)^{a+2}}.
$$
For $a=1$, the region in $\alpha$, $\beta$ plane such that $\hat{w}_{1,\alpha,\beta}(x)$ has only real negative zeros is shown in Figure~\ref{fig:1}.
\begin{figure}[!ht]
\begin{center}
\includegraphics[width=9cm]{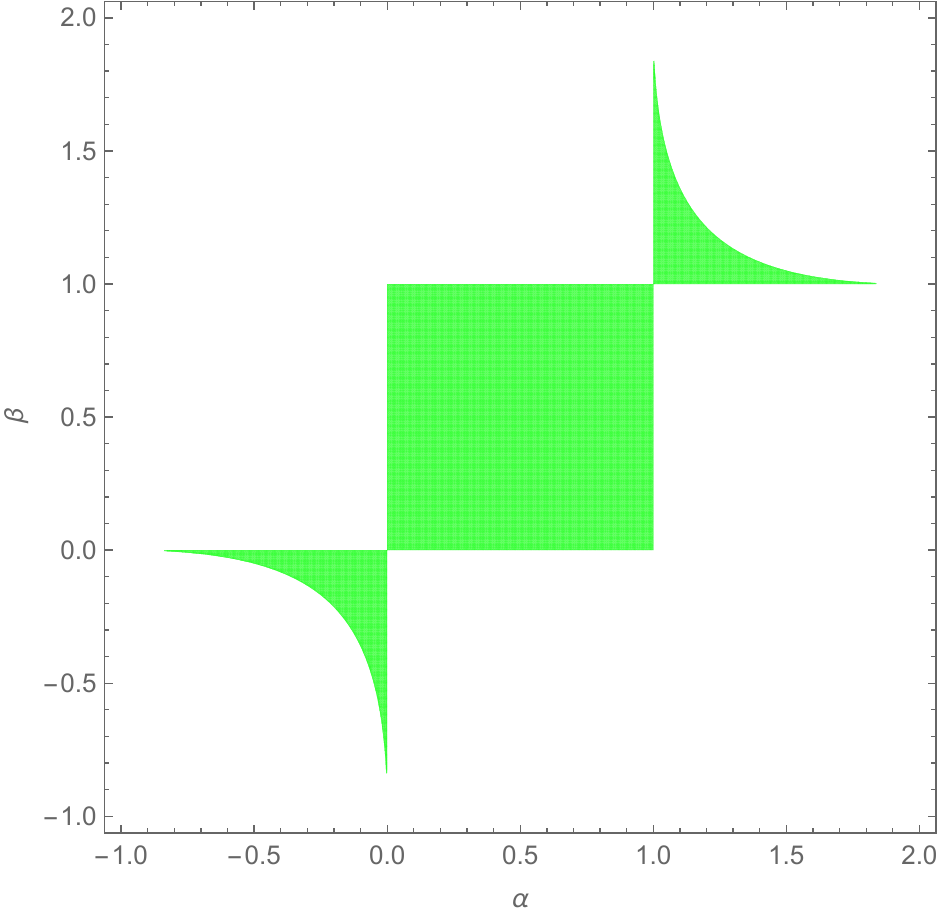}
\caption{The region in $\alpha,\beta$ plane such that $\hat{w}_{1,\alpha,\beta}(x)$ has only real negative zeros}
\label{fig:1}
\end{center}
\end{figure}
It is easy to see that part of the domain outside of the unit square corresponds to the values of $\alpha$, $\beta$ which do not satisfy Theorem~\ref{th:negative_zeros} (and thus also Brenti's theorem).

The following important result due to David\:G.\:Wagner \cite[Theorem~0.2]{wagner} allows generating more examples of $TP$ sequence of the form $\{F_m(n)\}_{n\ge0}$ with $F_m$ of any degree $m$.

\begin{theorem*}[Wagner]
Let $f, g \in \mathbb{R}[x]$ be real polynomials such that
the polynomials $\hat{w}_f$, $\hat{w}_g$ generated by
$$
\sum_{n=0}^\infty f(n) x^n=\frac{\hat{w}_f(x)}{(1-x)^{\mathrm{deg}(f)+1}}~\text{and}~
\sum_{n=0}^\infty g(n) x^n=\frac{\hat{w}_g(x)}{(1-x)^{\mathrm{deg}(g)+1}}$$
have only real negative zeros. Then $\hat{w}_{fg}$ generated by 
$$\sum_{n=0}^\infty f(n) g(n) x^n=\frac{\hat{w}_{fg}(x)}{(1-x)^{\mathrm{deg}(f)+\mathrm{deg}(g)+1}}
$$
also has only real negative zeros. 
\end{theorem*}

This leads to the following extensions of the previous examples.
\begin{example}\label{ex3}
From \eqref{exxx2} and  Wagner's theorem we conclude that for all $s\in\N$ all zeros of the polynomials $\hat{w}_{2s}$ of degree $2s$ generated by  
$$
\sum_{n=0}^{\infty} \Big(n^2 +\frac{1}{8}\Big)^s x^n= 
\frac{\hat{w}_{2s}(x)}{(1-x)^{2s+1}},
$$
are real and negative. Hence, the sequence $\left\{(n^2+1/8)^s\right\}_{n=0}^\infty$ is totally positive. 
\end{example}

\begin{example}\label{ex4}
From \eqref{exxx1} and  Wagner's Theorem we conclude that for all $s\in\N$ all zeros of the polynomials $\widebar{w}_{2s}$ of degree $2s$ defined by 
$$
\sum_{n=0}^{\infty}\Big(n+\frac{11}{10}\Big)^{2s}x^n=\frac{\widebar{w}_{2s}(x)}{(1-x)^{2s+1}},   
$$
are real and negative. Hence, the sequence $\left\{(n+11/10)^{2s}\right\}_{n=0}^\infty$ is totally positive.

 One can furthermore check that for every $s\in\N_{0}$ the polynomial $\widebar{w}_{2s+1}$ defined by
$$
\sum_{n=0}^\infty \Big(n+\frac{11}{10}\Big)^{2s+1} 
x^n= \frac{\widebar{w}_{2s+1}(x)}{(1-x)^{2s+2}},   
$$ 
has either nonnegative or complex zeros.  For example, 
$$
\sum_{n=0}^\infty \Big(n+\frac{11}{10}\Big)x^n=\frac{-x +11}{10(1-x)^2},
$$  
and  
$$
\sum_{n=0}^\infty \Big(n+\frac{11}{10}\Big)^3 x^n= \frac{-x^3 +\ldots+ \alpha}{1000(1-x)^4},~~ \alpha >0.
$$
Hence, the sequence $\left\{(n+11/10)^{2s+1}\right\}_{n=0}^{\infty}$ is not totally positive.
\end{example}  

To the best of our knowledge, the following potential extension of Wagner's theorem remains open.
\begin{open problem}\label{pr2}
Fix $a>0$ and let $f, g \in \mathbb{R}[x]$ be real polynomials such that
the polynomials $\hat{w}_f$, $\hat{w}_g$ generated by
$$
\sum_{n=0}^\infty f(n)\frac{(a)_n}{n!} x^n=\frac{\hat{w}_f(x)}{(1-x)^{\mathrm{deg}(f)+a}}~\text{and}~
\sum_{n=0}^\infty g(n)\frac{(a)_n}{n!}x^n=\frac{\hat{w}_g(x)}{(1-x)^{\mathrm{deg}(g)+a}}
$$
have only real negative zeros. Is it true that $\hat{w}_{fg}$ generated by 
$$
\sum_{n=0}^\infty f(n)g(n)\frac{(a)_n}{n!} x^n=\frac{\hat{w}_{fg}(x)}{(1-x)^{\mathrm{deg}(f)+\mathrm{deg}(g)+a}}
$$
also has only real negative zeros? 
\end{open problem}

\subsection{Positive  and mixed zeros}
\begin{lemma}\label{lm:Talpha-negative_s}
		Suppose $f$ is given by \eqref{eq:f-Pn}
		with the polynomial $P_n$ of degree $n$ having all its zeros in $(0,1)$. Then for all $s<0$ and 
		$\alpha>0$ the function $M_{\alpha}f$ has the form \eqref{eq:Talpha-f}, where the polynomial 
		$\hat{P}(x)$ has degree $n+1$ and all zeros in $(0,1)$.  Moreover,  
		\begin{equation}\label{eq:interlacing01}
			0<\hat{x}_{1}\le x_{1}\le\hat{x}_{2}\le\cdots\le \hat{x}_n\le x_{n}\le\hat{x}_{n+1}<1,
		\end{equation}
		where $x_j$ \emph{(}$\hat{x}_j$\emph{)} are zeros of $P_n$ \emph{(}$\hat{P}$\emph{)}.
		
		If $s>0$, $\alpha>0$  and  $s-\alpha-n<0$, $\hat{P}(x)$ has degree $n+1$ and all zeros but one in $(0,1)$.  More precisely,  
		\begin{equation}\label{eq:interlacing01-spos}
			0<\hat{x}_{1}\le x_{1}\le\hat{x}_{2}\le\cdots\le \hat{x}_n\le x_{n}<1<\hat{x}_{n+1},
		\end{equation}

		Moreover, if all zeros of $P_n(x)$ are simple, then so are all the zeros of $\hat{P}(x)$ and all inequalities in \eqref{eq:interlacing01}, \eqref{eq:interlacing01-spos} are strict.
	\end{lemma}
	
	\begin{proof}
		The proof is again based on formula \eqref{eq:hatP}.  From $s<0$ and $\alpha>0$ we have $s-\alpha-n<0$ and the leading term of $\hat{P}$ given by $a_{n}(s-\alpha-n)x^{n+1}$ cannot vanish, so that $\hat{P}$ has degree $n+1$. We again assume that all zeros are simple and $P_n(0)>0$ (otherwise multiply \eqref{eq:f-Pn} by $-1$), so that also $\hat{P}(0)=\alpha{P_n(0)}>0$. To show that all zeros of $\hat{P}$ lie in $(0,1)$ we use the argument similar to that in the proof of Lemma~\ref{lm:Talpha-positive_s}. Indeed, if $P_n(x_1)=0$, then $\hat{P}(x_1)<0$ since $P_n'(x_1)<0$, resulting in $0<\hat{x}_1<x_1$. Similarly, $(-1)^{j}\hat{P}(x_{j})>0$ for $j=2,\ldots,n$, and we have established the chain of inequalities in \eqref{eq:interlacing01} from $0$ to $x_n$.  Finally, from \eqref{eq:hatP} $\hat{P}(1)=sP_n(1)$ implying that there is one more zero of $\hat{P}$ in the interval $(x_n,1)$.  Indeed if $P_n'(x_n)>0$, then $P_n(1)>0$, while $\hat{P}(x_n)>0$ and $\hat{P}(1)<0$ and, similarly, if $P_n'(x_n)<0$.

		Suppose now that  $s>0$, $\alpha>0$  and  $s-\alpha-n<0$, then the above argument remains valid expect for the ultimate step.  Now  if $P_n'(x_n)>0$, so that   $P_n(1)>0$, we get  $\hat{P}(x_n)>0$ and  $\hat{P}(1)=sP_n(1)>0$ which implies that there is an even number of zeros of   $\hat{P}(x)$ in $(x_n,1)$. This even number must be zero as only one zero of $\hat{P}(x)$ is not in $(0,x_n)$.  At the same time, $\hat{P}(x)<0$ for sufficiently large $x$ as the leading coefficient  
		$a_{n}(s-\alpha-n)<0$ ($a_n>0$ because $P_n(1)>0$ and there are no zeros of $P_n(x)$ on $[1,\infty)$). Hence, $\hat{x}_{n+1}>1$.  If, on the other hand,  
		$P_n'(x_n)<0$, so that   $P_n(1)<0$, we get  $\hat{P}(x_n)<0$ and  $\hat{P}(1)=sP_n(1)<0$ which implies that there is an even number of zeros of   $\hat{P}(x)$ on $(x_n,1)$.   This even number must be zero as only one zero of $\hat{P}(x)$ is not in $(0,x_n)$.  At the same time, $\hat{P}(x)>0$ for sufficiently large $x$ as the leading coefficient  
		$a_{n}(s-\alpha-n)>0$ ($a_n<0$ because $P_n(1)<0$ and there are no zeros of $P_n(x)$ in $[1,\infty)$). Hence, $\hat{x}_{n+1}>1$.
	\end{proof}
	
		Recursive application of the Lemma~\ref{lm:Talpha-negative_s} to decomposition \eqref{eq:Tproduct} yields  the following theorem. 
	\begin{theorem}\label{th:zeros-in-0-1} 
	Suppose $a+m<1$, and $\alpha_j>0$ for $j=1,\ldots,m$. Then  all zeros of the $F_m$-Eulerian polynomial $\hat{w}(a;F_m\vert\:x)$ defined in \eqref{eq:fa-Euler}, \eqref{eq:fa-Euler1} or, equivalently, of  the hypergeometric function on the left hand side of \eqref{eq:fa-Euler1} are real and lie in the interval $(0,1)$.  If $1\le a+m<2$, then all  zeros  but one lie in $(0,1)$ plus one zero lies in $[1,\infty)$. 
	\end{theorem}
	
	\begin{proof}
		Indeed, retaining notation from the proof of Theorem~\ref{th:negative_zeros} under the assumption $a+m<1$ we have $s_{0}, s_{1},\ldots, s_{m-1}<0$ and the claim follows by recursive application of Lemma~\ref{lm:Talpha-negative_s} to decomposition \eqref{eq:Tproduct}.  If $1<a+m<2$ we have 
		 $s_{0}, s_{1},\ldots, s_{m-2}<0$, so that after $m-1$ steps the numerator polynomial has all its zeros in $(0,1)$ by Lemma~\ref{lm:Talpha-negative_s}. If now  $0<s_{m-1}<1$, $s_{m-1}-\alpha_m-(m-1)<0$ and the conclusion follows by \eqref{eq:interlacing01-spos}. If $a+m=1$, we have $s_{m-1}=0$ and one zero of  $\hat{w}(a;F_m\vert\:x)$ is equal to $1$.
	\end{proof}
		
	\begin{lemma}\label{lm:Talpha-zeros-with-gap}
		Suppose $f$ in given by \eqref{eq:f-Pn}
		with the polynomial $P_n$ of degree $n$ having all real zeros none of which lie in  $[0,1]$.  Suppose further that $s$ and $\alpha$ have the same sign while $s-n-\alpha>0$. Then the polynomial $\hat{P}(x)$ has degree $n+1$ and all its zeros are real and lie outside $[0,1]$. The zeros of $P_n$ and $\hat{P}$ interlace.   
		
		If $s,\alpha>0$, then the number of zeros of $P_n$ in $(1,\infty)$ is preserved by $\hat{P}$.
		
		If $s,\alpha<0$, then the number of zeros of $P_n$ in $(-\infty,0)$ is preserved by $\hat{P}$.
		
		If $s,\alpha$ have opposite signs, all zeros of $\hat{P}$ are still real, with one zero lying in $(0,1)$, and the number of zeros in $(-\infty,0)$ and $(1,\infty)$ coinciding with those for $P_n$. 
		
		Moreover, if all zeros of $P_n(x)$ are simple, then so are all the zeros of $\hat{P}(x)$.
	\end{lemma}
	
	\begin{proof}
		As	$s-\alpha-n>0$ the leading term of $\hat{P}$ given by $a_{n}(s-\alpha-n)x^{n+1}$ cannot vanish, so that $\hat{P}$ has degree $n+1$.  We can also assume without loss of generality that $a_{n}>0$.
		Suppose first that the zeros of $P_n$ are simple and ordered as follows:
		$$
		x_{1}<x_{2}<\cdots<x_{k}<0<1<x_{k+1}<\cdots<x_{n}.
		$$	
		The proof will be divided in two cases.
		
		\textbf{Case I}. $s,\alpha>0$, $s-n-\alpha>0$.  If $P_n(0)>0$, then $P_{n}'(x_{k})>0$ implying that $\hat{P}(0)=\alpha P_n(0)>0$ and  $\hat{P}(x_{k})<0$ by \eqref{eq:hatP} and hence $\hat{P}(x)$ has a zero in $(x_{k},0)$.  Similarly, if  $P_n(0)<0$, then $P_{n}'(x_{k})<0$ implying that $\hat{P}(0)=\alpha P_n(0)<0$ and  $\hat{P}(x_{k})>0$ by \eqref{eq:hatP} and hence $\hat{P}(x)$ again has a zero in $(x_{k},0)$.   Next, for $j=1,\ldots,k-1$ we have $\sign(P_{n}'(x_{j})P_{n}'(x_{j+1}))=-1$, so that   $\sign(\hat{P}(x_{j})\hat{P}(x_{j+1}))=-1$ by \eqref{eq:hatP} and there is a zero of $\hat{P}$ in each interval $[x_{j},x_{j+1}]$, $j=1,\ldots,k-1$. 
		As the leading coefficient of $\hat{P}$ is positive, $a_{n}(s-\alpha-n)>0$, for  $x\to-\infty$ we have 
		$$
		\sign(\hat{P}(-\infty))=(-1)^{n+1}=-\sign{P_{n}(-\infty)},
		$$
		which implies that $\hat{P}(x)$ has a zero in $(-\infty,x_1)$. Indeed, if $P_n(-\infty)>0$, then $P_n'(x_1)<0$, so that $\hat{P}(x_1)>0$ while $\hat{P}(-\infty)<0$; similarly, if $P_n(-\infty)<0$, then $P_n'(x_1)>0$, so that $\hat{P}(x_1)<0$, while $\hat{P}(-\infty)>0$.  
		Hence, there are at least $k+1$ zeros of $\hat{P}$ in $(-\infty,0)$. 
		
		In a similar fashion, for $j=k,\ldots,n-1$ we have 
		$$
		\sign(P_{n}'(x_{j})P_{n}'(x_{j+1}))=-1~\Rightarrow~\sign(\hat{P}(x_{j})\hat{P}(x_{j+1}))=-1
		$$
		by \eqref{eq:hatP} and there is a zero of $\hat{P}$ in each interval $[x_{j},x_{j+1}]$, $j=k+1,\ldots,n-1$ totaling at least $n-k-1$ zeros of $\hat{P}$ in $(x_{k+1},x_{n})$.  Hence, there are at least $n$ zeros of $\hat{P}$ in $(-\infty,0)\cup(x_{k+1},x_{n})$. Finally, there is one zero in $(x_{n},\infty)$, as $P_n(\infty)$ is positive and so is $\hat{P}(\infty)$, but $P_n'(x_n)>0$ implying that $\hat{P}(x_n)<0$.

		\textbf{Case II}. $s,\alpha<0$, $s-n-\alpha>0$.  If $P_n(1)>0$, then $P_{n}'(x_{k+1})<0$ implying that $\hat{P}(1)=s P_n(1)<0$ and  $\hat{P}(x_{k+1})>0$ by \eqref{eq:hatP} and hence $\hat{P}(x)$ has a zero in $(1,x_{k+1})$.  Similarly, if  $P_n(1)<0$, then $P_{n}'(x_{k+1})>0$ implying that $\hat{P}(1)=s P_n(1)>0$ and  $\hat{P}(x_{k+1})<0$ by \eqref{eq:hatP} so that $\hat{P}(x)$ again has a zero in $(1, x_{k+1})$.   Next, for $j=k+1,\ldots,n-1$ we have $\sign(P_{n}'(x_{j})P_{n}'(x_{j+1}))=-1$, so that   $\sign(\hat{P}(x_{j})\hat{P}(x_{j+1}))=-1$ by \eqref{eq:hatP} and there is a zero of $\hat{P}$ in each interval $[x_{j},x_{j+1}]$, $j=k+1,\ldots,n-1$.  Further,  there is one zero of $\hat{P}$ in $(x_{n},\infty)$. Indeed, $P_n(\infty)$ is positive and the leading coefficient $a_{n}(s-\alpha-n)$ of $\hat{P}$ is positive, so that  $\hat{P}(\infty)>0$, but $P_n'(x_n)>0$ implying that $\hat{P}(x_n)<0$ \eqref{eq:hatP}. Hence, there are at least $n-k+1$ zeros of $\hat{P}(x)$ in $(1,\infty)$.  
		
		In a similar fashion, for $j=1,\ldots,k-1$ we have 
		$$
		\sign(P_{n}'(x_{j})P_{n}'(x_{j+1}))=-1~\Rightarrow~\sign(\hat{P}(x_{j})\hat{P}(x_{j+1}))=-1
		$$
		by \eqref{eq:hatP} and there is a zero of $\hat{P}(x)$ in each interval $[x_{j},x_{j+1}]$, $j=1,\ldots,k-1$ totaling at least $k-1$ zeros of $\hat{P}$ in $(x_{1},x_{k})$.  Hence, there are at least $n$ zeros of $\hat{P}$ in $(x_{1},x_{k})\cup(1,\infty)$. Finally, there is one zero in $(-\infty,x_{1})$. Indeed, $\sign(P_n(-\infty))=-\sign(\hat{P}(-\infty))$ because both polynomials have positive leading coefficient. Hence, when $P_n(-\infty)>0$, $P_n'(x_1)<0$ and so $\hat{P}(x_1)>0$, while $\hat{P}(-\infty)<0$; similarly, when  $P_n(-\infty)<0$, $P_n'(x_1)>0$ and so $\hat{P}(x_1)<0$, while $\hat{P}(-\infty)>0$.  
	\end{proof}
	
	Recursive application of the Lemma~\ref{lm:Talpha-zeros-with-gap} to decomposition \eqref{eq:Tproduct} yields  the following theorem. 
	\begin{theorem}\label{th:zeros_greater_1} 
		Suppose $m\ge2$ and  $\alpha_j<a<1-m$ for $j=1,\ldots,m$. Then all zeros of the $F_m$-Eulerian polynomial $\hat{w}(a;F_m\vert\:x)$ defined in \eqref{eq:fa-Euler}, \eqref{eq:fa-Euler1} or, equivalently, of  the hypergeometric function on the left hand side of \eqref{eq:fa-Euler1} are real and lie in the interval $(1,\infty)$.  If $1-m<a<2-m$ and  $\alpha_j<a$ for $j=1,\ldots,m$, then all zeros but one  lie in $(1,\infty)$ plus one zero lies in $(0,1)$. 
	\end{theorem}
	
	\begin{proof}
		Indeed, retaining notation from the proof of Theorem~\ref{th:negative_zeros} under the assumption $\alpha_j<a<1-m$ for $j=1,\ldots,m$ we have $s_0,s_1,\ldots,s_{m-1}<0$, and $s_k-n_k=a>\alpha_{k+1}$ for $k=0,\ldots,m-1$.  As we start with constant polynomial $P_0=1$ without zeros in $(-\infty,0)$, then by  Lemma~\ref{lm:Talpha-zeros-with-gap} in each step the number of zeros in $(1,\infty)$ is increased by $1$, so that all zeros of $\hat{w}(a;F_m\vert\:x)$ lie in $(1,\infty)$.   If $1-m<a<2-m$, then  $s_0,s_1,\ldots,s_{m-2}<0$ and $s_{m-1}>0$ while $\alpha_{m}<0$, so that the conclusion follows by the third claim of Lemma~\ref{lm:Talpha-zeros-with-gap}.
	\end{proof}

\section{Hypergeometric representations and change of basis}
Applying the Cauchy product to the generating relation \eqref{eq:fa-Euler1} and manipulating the Pochhammer symbols, we obtain:
\begin{equation}\label{eq:wh-defined}
\hat{w}\big(a;F_m\big\vert\:x\big)=(1-x)^{a+m}{}_{r+1}F_{r}\left.\!\!\left(\!\begin{matrix}a, \f+\m\\\f\end{matrix}\right\vert x\right)
=\sum\limits_{k=0}^{m}\frac{x^k}{k!}(-m-a)_{k}\cdot
{}_{r+2}F_{r+1}\!\left(\begin{matrix}-k,a,\f+\m\\1+m-k+a,\f\end{matrix}\right),
\end{equation}
where we omitted the argument $1$ from the notation of the generalized hypergeometric function on the right hand side -- a convention to be used from here onward.  
This implies that
\begin{equation}\label{eq:hatwPm}
\hat{w}\big(a;F_m\big\vert\:x\big)=\frac{(-1)^m}{m!}
	\sum\limits_{n=0}^{m}\frac{(-m)_n}{n!}
	\hat{P}_{m}(n)x^n=\frac{(-1)^m}{m!}F\!\left(\begin{matrix}-m\\{-} \end{matrix}\:\bigg\vert\:\hat{P}_m\bigg\vert\: x\right),
\end{equation}
where  $\hat{P}_{m}(t)$ is  a polynomial of degree $m$
satisfying the interpolation conditions
$$
\hat{P}_{m}(n)=\frac{(-m-a)_{n}}{(-m)_{n}}
{}_{r+2}F_{r+1}\!\left(\begin{matrix}-n,a,\f+\m\\1+m-n+a,\f\end{matrix}\right), ~~n=0,\ldots,m.
$$
The polynomial $\hat{P}_{m}(t)$ can be constructed using Lagrange's or Newton's formula with equidistant nodes as given, for instance, in  \cite[Lemma~1]{Karp2025}.  Lagrange's form  according to \cite[(2)]{Karp2025} reads (after some simplification)
$$
\hat{P}_{m}(t)=\sum\limits_{k=0}^{m}\binom{m+a}{k}(-t)_k(t-m)_{m-k}\cdot{}_{r+2}F_{r+1}\!\left(\begin{matrix}-k,a,\f+\m\\1+m-k+a,\f\end{matrix}\right).
$$
We will show now that there exists a polynomial of degree $m-\max(m_1,\ldots,m_r)$ satisfying the same interpolation conditions. This is achieved through connecting the   
generating relation \eqref{eq:fa-Euler1} with the generalized Euler transformation due to Miller and Paris given by  \cite[Theorem~4]{MP2013}:  
\begin{equation}\label{eq:MP2general}
F\!\left(\begin{matrix}\delta,\epsilon\\\rho\end{matrix}\,\bigg\vert\,H_{\omega}\,\bigg\vert\, x\right)\!=\!{}_{r+2}F_{r+1}\left(\begin{matrix}\delta, \epsilon,\nuu+\omm\\\rho,\nuu\end{matrix}\,\bigg\vert\,x\right)
		\!=\!(1-x)^{\rho-\delta-\epsilon-\omega}F\left(\!\begin{matrix}\rho-\delta-\omega, \rho-\epsilon-\omega\\\rho\end{matrix}\,\bigg\vert\,\hat{Q}_{\omega}\,\bigg\vert\,x\right),
	\end{equation}    
where $\nuu=(\nu_1,\ldots,\nu_r)$ is a complex vector and $\omm=(\omega_1,\ldots,\omega_r)$ in a vector of positive integers, $\omega=\omega_1+\cdots+\omega_r$. Formula \eqref{eq:MP2general} is valid for $(\rho-\epsilon-\omega)_{\omega}\ne0$, $(\rho-\delta-\omega)_{\omega}\ne0$ and $(1+\delta+\epsilon-\rho)_{\omega}\ne0$. Here
$$
H_{\omega}(t)=\frac{(\nuu+t)_{\omm}}{(\nuu)_{\omm}},
$$
similarly to the notation in \eqref{eq:Fm-defined} and \eqref{eq:fa-Euler1}.  The characteristic polynomial $\hat{Q}_{\omega}$, also of degree $\omega$, can be written in a variety of different, but equivalent forms given in  \cite[(19), (20), (24), (25)]{Karp2025}.  To unify and conceptualize those forms the first author introduced in \cite[Section~2.2]{Karp2025} the notion of the second Miller-Paris operator $\hat{T}_{\omega}(\delta,\epsilon;\rho)$ which is precisely the mapping from $H_{\omega}$ to $\hat{Q}_{\omega}$ in \eqref{eq:MP2general}:  $\hat{Q}_{\omega}=\hat{T}_{\omega}(\delta,\epsilon;\rho)H_{\omega}$.  We will cite some of its explicit forms below as needed. 

Denote by $\a_{[i]}$  vector $\a$ with deleted $i$-th component. If we put in \eqref{eq:MP2general}  $\rho=f_1$, $\delta=a$, $\epsilon=f_1+m_1$, $\nuu=\f_{[1]}$, $\omm=\m_{[1]}$, $\omega=m_2+\cdots+m_r$ and compare it with   \eqref{eq:fa-Euler1}, we obtain an expansion of $\hat{w}$ in monomial basis: 
\begin{proposition}\label{pr:fEuler-monomial}
Suppose, without loss of generality, that $m_1=\max(m_1,\ldots,m_r)$ and put $\omega=m-m_1$.  Then assuming $a>0$ and  $(f_1-a-\omega)_{\omega}\ne0$ we have
\begin{equation}\label{eq:wh-formula}
\hat{w}\big(a;F_m\big\vert\:x\big)=\sum\limits_{n=0}^{m}\frac{(-m)_{n}(f_1-a-\omega)_{n}}{(f_1)_{n}n!}\hat{Q}_{\omega}(n)x^n
={}_{\omega+2}F_{\omega+1}\left(\!\begin{matrix}-m, f_1-a-\omega, \etta+1 \\f_1,\etta\end{matrix}\,\bigg\vert\,x\right),
\end{equation} 
where
$$
\big(\f_{[1]}\big)_{\m_{[1]}}\hat{Q}_{\omega}(t)=\Big[\hat{T}_{\omega}(a,f_1+m_1;f_1)\big(\f_{[1]}+\cdot\big)_{\m_{[1]}}\Big](t)
$$
and  $\etta=(\eta_1,\ldots,\eta_{\omega})$ are the roots of $\hat{Q}_{\omega}(-t)$.  If $a<0$ formula \eqref{eq:wh-formula} still holds under additional assumption $(1+a+m_1)_{\omega}\ne0$.
\end{proposition}
Note that the condition $(\rho-\epsilon-\omega)_{\omega}\ne0$ required for \eqref{eq:MP2general} holds automatically and  $(1+\delta+\epsilon-\rho)_{\omega}\ne0$ holds automatically  if $a=\delta>0$, while the condition $(\rho-\delta-\omega)_{\omega}\ne0$ becomes  $(f_1-a-\omega)_{\omega}\ne0$.

Comparing \eqref{eq:wh-formula} with \eqref{eq:hatwPm} we conclude that
$$
\hat{P}_{m}(n)=\frac{(f_1-a-\omega)_{n}}{(f_1)_{n}}\hat{Q}_{\omega}(n).
$$
Hence, the degree of the polynomial interpolating the coefficients of $\hat{w}(x)$ is reduced by $m_1$, where we can choose 
$m_1=\max(m_1,\ldots,m_r)$ due to symmetry.  

Starting with the left hand side of \eqref{eq:fa-Euler1} we can naturally define another polynomial, which we denote by $w$, as follows. Using Euler's transformation \cite[(4)]{Karp2025}, we  get
\begin{equation}\label{eq:MP1polynomial}
	(1-x)^{a}{}_{r+1}F_{r}\!\!\left(\!\begin{matrix}a, \f+\m\\\f\end{matrix}\bigg\vert\,x\right)
	=\sum\limits_{k=0}^{m}\frac{(a)_k}{k!}
	{}_{r+1}F_{r}\!\left(\begin{matrix}-k,\f+\m\\\f\end{matrix}\right)\Big(\frac{x}{x-1}\Big)^k=:w\Big(a;F_m\Big\vert\:\frac{x}{x-1}\Big),
\end{equation}
where the last equality is the definition of the polynomial $w$  of degree $m$ in the variable in $x/(x-1)$ (note that ${}_{r+1}F_{r}(-k,\f+\m;\f)$ vanishes for $k>m$ by \cite[Corollary~2]{MP2012}). This polynomial can be written as 
	$$
	w(y)=\sum_{k=0}^{m}\frac{(a)_k}{k!}P_m(k)y^k,
	$$
	where $P_m$ is the interpolating polynomial whose explicit form can be read off from any of the expressions given in  \cite[Lemma~1]{Karp2025}, for instance,
	$$
	P_{m}(t)=\frac{(-1)^m}{m!}\sum\limits_{k=0}^{m}\binom{m}{k}(-t)_k(t-m)_{m-k}\cdot{}_{r+1}F_{r}\!\left(\begin{matrix}-k,\f+\m\\\f\end{matrix}\right).
	$$
This polynomial is symmetric with respect to the parameters $f_1,\ldots,f_r$ and has degree $m$. As in the case of $\hat{w}$, there exists a lower degree polynomial that interpolates the coefficients of $w$. It can be found by using the first Miller-Paris transformation
\begin{equation}\label{eq:MP1general}
F\!\left(\begin{matrix}\delta,\epsilon\\\rho\end{matrix}\,\bigg\vert\,H_{\omega}\,\bigg\vert x\right)=(1-x)^{-\delta}F\!\left(\begin{matrix}\delta,\rho-\epsilon-\omega\\\rho\end{matrix}\,\bigg\vert\,Q_{\omega}\:\bigg\vert\, \frac{x}{x-1}\right)
\end{equation}
valid when $(\rho-\epsilon-\omega)_{\omega}\ne0$, 
where the characteristic polynomial $Q_{\omega}$  is defined by the action of the first Miller-Paris operator on the polynomial $H_{\omega}$: $Q_{\omega}=T_{\omega}(\epsilon;\rho)H_{\omega}$. Its various explicit forms are given in  \cite[(8)-(13)]{Karp2025}.  We can again put  $\delta=a$, $\rho=f_1$, $\epsilon=f_1+m_1$, $\nuu=\f_{[1]}$, $\omm=\m_{[1]}$, $\omega=m_2+\cdots+m_r$, $\rho-\epsilon-\omega=-m$ (so that $(\rho-\epsilon-\omega)_{\omega}\ne0$) in \eqref{eq:MP1general}. On comparing with \eqref{eq:MP1polynomial}  this  yields
	\begin{equation}\label{eq:MP1poly_reduced}
		w\big(a;F_m\big\vert\:y\big)=\sum\limits_{k=0}^{m}\frac{(a)_k(-m)_{k}}{(f_1)_{k}k!}
		Q_{\omega}(k)y^k=F\!\left(\begin{matrix}-m,a\\f_1\end{matrix}\bigg\vert\,Q_{\omega}\,\bigg\vert\, y\right),
	\end{equation}
	where $\omega=m-m_1<m$ and
\begin{equation}\label{eq:Q_omega}
	Q_{\omega}(t)=\frac{1}{\big(\f_{[1]}\big)_{\m_{[1]}}}\big[T_{\omega}(f_1+m_1; f_1)\big(\f_{[1]}+\cdot\big)_{\m_{[1]}}\big](t)
\end{equation}
	is a polynomial of degree $\omega$. It can be also found from the interpolation conditions  
	$$
	Q_{\omega}(k)=\frac{(f_1)_{k}}{(-m)_{k}}{}_{r}F_{r-1}\!\left(\begin{matrix}-k,\f_{[1]}+\m_{[1]}\\\f_{[1]}\end{matrix}\right),~~k=0,\ldots,\omega,
	$$
using Lagrange's or Newton's form \cite[Lemma~1]{Karp2025} or any of the formulas \cite[(8)-(13)]{Karp2025}. For instance, according to  \cite[(9)]{Karp2025}, 
$$
Q_{\omega}(t)=\sum_{k=0}^{\omega}\frac{(f_1+m_1)_k(-1)^k(-t)_{k}(t-m)_{\omega-k}}{(-m)_{\omega}k!}{}_{r}F_{r-1}\!\left(\begin{matrix}-k,\f_{[1]}+\m_{[1]}\\\f_{[1]}\end{matrix}\right).    
$$
As a by-product this implies the summation formula 
$$
	{}_{r+1}F_{r}\!\left(\begin{matrix}-k,\f+\m\\\f\end{matrix}\right)=\frac{(-m)_{k}}{(-m)_{m-m_1}(f_1)_{k}}
	\sum_{j=0}^{k}\frac{(-k)_{j}(k-m)_{m-m_1-j}(f_1+m_1)_j}{(-1)^jj!}{}_{r}F_{r-1}\!\left(\begin{matrix}-j,\f_{[1]}+\m_{[1]}\\\f_{[1]}\end{matrix}\right)
$$
	valid for all non-negative integers $k$ (in particular, it follows that the left hand side vanishes for $k>m$).  This formula is a terminating case of \cite[Theorem~4]{MP2012}.
	
	The values of $P_m$ and $Q_{\omega}$ are clearly related by 
	$$
	P_m(k)=\frac{(f_1)_{k}}{(-m)_k}Q_{\omega}(k), ~~~k=0,\ldots,m.
	$$
The above observation can be reformulated as follows:  suppose the polynomial $F_m(t)$ possesses  a sequence of $m_1$ roots separated by unity, then the degree of the polynomials interpolating the coefficients of  $\hat{w}$  and $w$ can be reduced to $m-m_1$  by writing them  in the forms as in \eqref{eq:wh-formula} and \eqref{eq:MP1poly_reduced}, respectively. 
Comparing \eqref{eq:MP1polynomial} and \eqref{eq:wh-defined} we arrive at 
\begin{proposition}\label{pr:hatw-Bernstein}
The expansion of $\hat{w}$ in the Bernstein basis  $x^k(1-x)^{m-k}$, $k=0,\ldots,m$, is given by
\begin{equation}\label{eq:wh-Bernstein}
\hat{w}\big(a;F_m\big\vert\:x\big)=(1-x)^{m} w\Big(a;F_m\Big\vert\:\frac{x}{x-1}\Big)=\sum\limits_{k=0}^{m}(-1)^k\frac{(a)_k(-m)_{k}}{(f_1)_{k}k!}
Q_{\omega}(k)x^k(1-x)^{m-k},
\end{equation}
where $\omega=m-m_1$ and $Q_{\omega}$ is given in \eqref{eq:Q_omega}. In other words, the coefficients of $w(y)$ in the monomial basis are equal to the coefficients of $\hat{w}(x)$ in the Bernstein basis and vice versa.    
\end{proposition}

Another expression for the coefficients of $\hat{w}$ in the Bernstein basis can obtained  by an application of the second Miller-Paris transformation \eqref{eq:MP2general} to $w(y)$.  Indeed, by \eqref{eq:MP1poly_reduced} and \eqref{eq:MP2general} we have 
	$$
	w\big(a;F_m\big\vert\:y\big)=F\!\left(\begin{matrix}-m,a\\ f_1\end{matrix}\Big\vert\:Q_{\omega}\Big\vert\: y\right)
	=(1-y)^{f_1+m_1-a}F\!\left(\begin{matrix}f_1-a-\omega,f_1+m_1\\ f_1\end{matrix}\Big\vert\:\hat{R}_{\omega}\Big\vert\: y\right),
	$$
	where 
	$$
	\hat{R}_{\omega}(t)=\hat{T}_{\omega}(-m,a;f_1)Q_{\omega}=T_{\omega}(-m,a;f_1)T_{\omega}(f_1+m_1;f_1)\frac{\big(\f_{[1]}+\cdot\big)_{\m_{[1]}}}{\big(\f_{[1]}\big)_{\m_{[1]}}}.
	$$
	According to \cite[(15), (28)]{Karp2025} we have 
	$$
	\hat{T}_{\omega}(-m,a;f_1)=T_{\omega}(-m;f_1)T_{\omega}(a;f_1)=T_{\omega}(a;f_1)\big[T_{\omega}(f_1+m_1;f_1)\big]^{-1}.
	$$
	Hence,
	$$
	\hat{R}_{\omega}(t)=\frac{1}{\big(\f_{[1]}\big)_{\m_{[1]}}}\Big[T_{\omega}(a;f_1)\big(\f_{[1]}+\cdot\big)_{\m_{[1]}}\Big](t).
	$$
	This leads to the connection formula which is a guise of the first Miller-Paris transformation written for generalized $f$-Eulerian polynomials \eqref{eq:fa-Euler}:
	\begin{equation}\label{eq:whF-whR-connection}
		(1-y)^{m}\hat{w}\Big(a;F_m\Big\vert\:\frac{y}{y-1}\Big)=w\big(a;F_m\big\vert\:y\big)=\hat{w}\Big(f_1-a-\omega;\frac{(f_1+\cdot)_{m_1}}{(\f)_{\m}}T_{\omega}(a;f_1)\big(\f_{[1]}+\cdot\big)_{\m_{[1]}}\Big\vert\:y\Big).
	\end{equation}  
	If we write $x=y/(y-1)$ and use definition \eqref{eq:fa-Euler1} for the right hand side we get another form of the expansion of $\hat{w}(a;F_m|\,x)$ in the Bernstein basis.    Formula \eqref{eq:whF-whR-connection} also allows viewing $w$ as (generalized) $f$-Eulerian polynomial $\hat{w}$ generated by the parameter $f_1-a-\omega$ and the polynomial
	$(f_1+\cdot)_{m_1}\hat{R}_{\omega}(t)/(f_1)_{m_1}$.

	\section{\texorpdfstring{$d$}{d}-Narayana polynomials}

    An example of generalized $f$-Eulerian polynomial, we will discuss in a greater detail in this paper, is $d$-Narayana polynomials defined by Robert Sulanke \cite{Sulanke2004} as follows. Let $\mathcal{C}(d,m)$ denote the set of lattice paths in $d$-dimensional Euclidean space with steps  
	$$
	X_1:=(1, 0,\ldots, 0),~X_2:=(0, 1,\ldots, 0),\ldots,~X_d:=(0, 0, . . . , 1),
	$$
	running from $(0,0,\ldots, 0)$ to $(m,m,\ldots, m)$ and lying in the chamber  
	$$
	\{(x_1,x_2,\ldots,x_d) : 0\le x_1 \le x_2 \le\cdots\le x_d\}
	$$ 
	(known as ballot paths for $d$ candidates).
	On a path $P:=p_1p_2 \cdots p_{dm}$, we call any step pair $p_{i}p_{i+1}$ an ascent (respectively, a descent) if $p_{i}p_{i+1}= X_jX_l$ with $j<l$ (respectively $j>l$).  For each path $P\in\mathcal{C}(d,m)$ put
	\begin{align*}
		&\mathrm{asc}(P)=|\{i: p_{i}p_{i+1}=X_jX_l~\text{with}~j<l\}|~\text{- number of ascents on}~P
		\\	
		&\mathrm{des}(P)=|\{i: p_{i}p_{i+1}=X_jX_l~\text{with}~j>l\}|~\text{- number of descents on}~P
	\end{align*}
	For any dimension $d\ge2$ and for $0\le k \le (d-1)(m-1)$, define the $d$-Narayana  number as
	$$
	N(d,m,k)=|\{P\in\mathcal{C}(d,m): \mathrm{asc}(P)=k\}|.
	$$
	Finally,
	\begin{equation}\label{eq:dNarayana}
		N_{d,m}(x)=\sum\limits_{k=0}^{(d-1)(m-1)}N(d,m,k)x^k    
	\end{equation}
	is called $d$-dimensional Narayana polynomial or $d$-Narayana polynomial.  Another combinatorial meaning of the numbers $N(d,m,k)$ is the number of rectangular standard Young tableaux with $d$ rows and $m$ columns having $k$ occurrences of an integer $i$ appearing in a lower row
	than that of $i + 1$ \cite[Remark~1.3]{Sulanke2004}.  Sulanke also established the connection with order polynomials $\Omega_{P,\omega}(m)$ mentioned above \cite[section~2]{Sulanke2004}. The relation for $N(d,m,k)$ found by Sulanke \cite[Proposition~1]{Sulanke2004} and important for us is the following:
	$$
	N(d,m,k)=\sum\limits_{j=0}^{k}(-1)^{k-j}\binom{dm+1}{k-j}\prod\limits_{i=0}^{d-1}\binom{m+i+j}{m}\binom{m+i}{m}^{-1}.  
	$$
	It has been observed by Johann Cigler \cite{Cigler} that by  elementary manipulations this formula can be written as
	\begin{equation}\label{eq:SulankeNdmk}
		N(d,m,k)=\sum\limits_{j=0}^{k}\frac{(-dm-1)_{k-j}}{(k-j)!}\left\langle\!\!\! \begin{array}{c}m+j\\m\end{array}\!\!\!\right\rangle_{\!\!d},     
	\end{equation}
	where 
	$$
	\left\langle\!\!\! \begin{array}{c}m+j\\m\end{array}\!\!\!\right\rangle_{\!\!d}=\prod_{k=1}^{d}\frac{(m+k)_{j}}{(k)_{j}}
	$$
	is called $d$-Hoggatt's binomial counting the number of semi-standard Young tableaux with shape $d^m$ (a box with $d$ columns and $m$ rows) with entries from $\{1,2,\ldots,m+j\}$. Cigler \cite[Theorem~1]{Cigler} established that Hankel determinants of binomial coefficients can be expressed in terms of Hoggatt's binomials as follows: 
	$$
	\left\langle\!\!\! \begin{array}{c}k\\m\end{array}\!\!\!\right\rangle_{\!\!d}=(-1)^{d(d-1)/2}\det\Big[\binom{k+i+j}{m+d-1}\Big]_{i,j=0}^{d-1}.
	$$
	Sulanke's formula \eqref{eq:SulankeNdmk}
	immediately implies that $d$-Narayana polynomials $N_{d,m}$ can be generated by \eqref{eq:CiglerF}.	Another interpretation of the polynomial $N_{d,m}(x)$ and of the hypergeometric series on the left hand side of \eqref{eq:CiglerF} was presented more recently by Luke Braun \cite{Braun2019}. Namely, he showed in \cite[Theorem~1]{Braun2019} that the $d$-Narayana polynomial $N_{d,m-d}(x)$ is the so-called $h$-polynomial of the Grassmanian $\mathrm{Gr}(d,m)$, while the corresponding  hypergeometric function of the left hand side is its Hilbert series.

In a letter to the first author	dated February 28, 2022, Johann Cigler asked whether certain properties of $d$-Narayana polynomials can be established 
only by using the hypergeometric generating function \eqref{eq:CiglerF}.   Namely, the sequence $(N_{d,m}(1))_{m\ge0}$ known as the multidimensional Catalan
	numbers, \cite[A060854]{Sloane}, is  given by 
	\begin{equation}\label{eq:d-Catalan}
		N_{d,m}(1)=(md)!\prod_{j=0}^{d-1}\frac{j!}{(m+j)!},    
	\end{equation}
	see \cite[art. 93-103]{MacMahon}, \cite[p.2]{Sulanke2004}.  Next, following Sulanke \cite[Corollary~1]{Sulanke2004}, the polynomial $N_{d,m}(x)$ is palindromic (or self-reciprocal) and has degree $(d-1)(m-1)$, $N_{d,m}(x)=x^{(d-1)(m-1)}N_{d,m}(1/x)$. Finally, as we mentioned in the introduction, Agapito in \cite[Theorem~3.13]{Agapito} and later Chen, Yang and Zhang in \cite[Theorem 3.1]{CYZ} proved  that all zeros of  $N_{d,m}(x)$  are real and negative, resolving and strengthening a previous conjecture by Kirillov claiming that the sequence of $d$-Narayana numbers is log-concave.  In this section, we will prove all the above facts only by using the generating function   \eqref{eq:CiglerF}. 

Comparing definition \eqref{eq:Narayana_hyper} with \eqref{eq:f-Euler} we see that the corresponding polynomial 
$$
f(n)=\frac{(m+1)_n\cdots(m+d)_{n}}{(1)_{n}\cdots(d)_{n}}
$$
has degree $md$, so that the general theory suggests that the degree of $N_{d,m}$ should be $md$. However, choosing $a=m+d$ in the definition  \eqref{eq:fa-Euler1} and using the relation $(f+n)_{m}/(f)_{m}=(f+m)_{n}/(f)_{n}$ the generating function \eqref{eq:CiglerF} implies that 
\begin{subequations}\label{eq:dNarayanaFEuler}
\begin{align}
	N_{d,m}(x)=&\hat{w}(m+d; F_{K}\vert\:x), \label{eq:Nw} 
\end{align}
where the polynomial $F_{K}$ has degree $K=(d-1)(m-1)$ and   is given by
\begin{align}
  F_{K}(t)=&\frac{(2+t)_{m-1}(3+t)_{m-1}\cdots(d+t)_{m-1}}{(2)_{m-1}(3)_{m-1}\cdots(d)_{m-1}}.  \label{eq:FM}
\end{align}
\end{subequations}
This implies that the polynomial $N_{d,m}(x)$ has degree at most  $K$.

\subsection{Expansions in monomial and Bernstein bases}
 In this subsection we will derive an explicit formula for $N_{d,m}(x)$ by applying the second Miller-Paris transformation \eqref{eq:MP2general} to the function on the left hand side of \eqref{eq:CiglerF} with identifications $\delta=m+d$, $\epsilon=m+1$, $\rho=2$, $\omega=L:=(d-2)(m-1)$ and 
\begin{equation}\label{eq:HM}
H_{L}(x)=\frac{(3+x)_{m-1}\cdots(d+x)_{m-1}}{(3)_{m-1}\cdots(d)_{m-1}},
\end{equation}
where $L$ is the degree of $H_{L}$. This yields:
\begin{equation}\label{eq:Narayana_hyper}
{}_dF_{d-1}\!\left(\!\!\begin{array}{c}m+1,m+2,\ldots,m+d\\2,3,\ldots,d\end{array}\bigg\vert\,x\right)\!=\!
	(1-x)^{-md-1}F\!\left(\!\begin{matrix}-K, -M\\2 \end{matrix}\,\bigg\vert\,\hat{Q}_{L}\,\bigg\vert\, x\right)\!=\!\frac{N_{d,m}(x)}{(1-x)^{md+1}},    
\end{equation}
where $M=(d-1)m$ and the characteristic polynomial $\hat{Q}_{L}(t)=[\hat{T}_{L}(m+1,m+d;2)H_{L}(\cdot)](t)$ with $\hat{T}_{L}(m+1,m+d;2)$ denoting the second Miller-Paris operator, see \cite[Section~2.2]{Karp2025} and text below \eqref{eq:MP2general}. 
This shows that the degree of $N_{d,m}$ is indeed equal to $K=(d-1)(m-1)$.   Using the form \cite[Theorem~2]{KPChapter2020} or \cite[(24)]{Karp2025} for the  characteristic polynomial, we will have 
\begin{equation}\label{eq:hatQN}
\hat{Q}_{L}(t)=\!\!\!\sum\limits_{j=0}^{L}\frac{(m+d)_j(-M-3)_{j}(-t)_{j}(t-M)_{L-j}}{(-1)^j(-M)_{L}(-K)_{j}j!}
	{_dF_{d-1}}\!\left(\begin{array}{c}-j,m+1,\ldots,m+d-1\\K+3-j,3,\ldots,d\end{array}\right).    
\end{equation}
Applying 
	$$
	\frac{(-K)_k(-M)_{j}}{(2)_{j}j!}=\frac{1}{K+1}\binom{M}{j}\binom{K+1}{j+1},
	$$
	we obtain 
\begin{theorem}\label{th:dNar-explicit}
The following formula holds for the polynomial $N_{d,m}(x)$:
\begin{equation}\label{eq:Narayana_explicit}
		N_{d,m}(x)=\frac{1}{K+1}\sum\limits_{j=0}^{K}\binom{M}{j}\binom{K+1}{j+1}\hat{Q}_{L}(j)x^j,
	\end{equation}
where $M=(d-1)m$, $K=(d-1)(m-1)$ and  $\hat{Q}_{L}(t)=[\hat{T}_{L}(m+1,m+d;2)H_{L}(\cdot)](t)$ with $H_{L}$ defined in \eqref{eq:HM} and the second Miller-Paris operator $\hat{T}_{L}(m+1,m+d;2)$ can be computed by any of the expressions \cite[(19), (20), (24), (25)]{Karp2025}. In particular, $\hat{Q}_{L}$ can be found by  \eqref{eq:hatQN}.
\end{theorem}
For $d=2$, we have $L=0$, $M=m$, $K=m-1$, so that $\hat{Q}_{L}\equiv1$ and we recover the standard expression for the Narayana polynomials (note that our definition follows \cite{Sulanke2004}, while in many other works $xN_{d,m}(x)$ corresponds to our $N_{d,m}(x)$).  The coefficient at $x^j$ is the $d$-Narayana number $N(d,m,j)$, see  \cite[(2),(3)]{Sulanke2004},  \cite[(12)]{Cigler}, so that the above furnishes an alternative  explicit expression for these numbers. 
We can also rewrite \eqref{eq:Narayana_hyper} in terms of the standard hypergeometric notation:
$$
N_{d,m}(x)={}_{L+2}F_{L+1}\left.\!\!\left(\!\begin{matrix}-K,-M, \hat{\etta}+1\\2, \hat{\etta}\end{matrix}\right\vert x\right),
$$
where  $\hat{\etta}=(\hat{\eta}_1,\ldots,\hat{\eta}_{L})$ are the roots of the polynomial $\hat{Q}_{L}(-t)$.  The above hypergeometric expression for $d$-Narayana polynomial was previously discovered by Luke Braun in \cite[Theorem~3]{Braun2019}. He also used the second Miller-Paris' transformation \cite[Theorem~4]{MP2013}  and Miller-Paris form \cite[(20)]{Karp2025} for the characteristic polynomial $\hat{Q}_{L}$.  The form \eqref{eq:Narayana_explicit} is not contained explicitly in his work. Note that that we can obtain various alternative expressions for $N_{d,m}(x)$ by making different choices for $\delta,\epsilon$ and $\rho$ when applying \eqref{eq:MP2general}.  For instance, we can choose $\delta=m+1$, $\epsilon=m+2$, $\rho=d$ to get 
\begin{equation*}
N_{d,m}(x)=F\!\left(\!\begin{matrix}-M+1, -M \\d \end{matrix}\,\bigg\vert\,\hat{P}_{\widebar{M}}\,\bigg\vert\, x\right),    
\end{equation*}
where  $\hat{P}_{\widebar{M}}(t)=[\hat{T}_{\widebar{M}}(m+1,m+2;d)G_{\widebar{M}}(\cdot)](t)$, $\widebar{M}=(d-2)(m+1)$ and 
$$
G_{\widebar{M}}(x)=\frac{(2+x)_{m+1}\cdots(d-1+x)_{m+1}}{(2)_{m+1}\cdots(d-1)_{m+1}}.
$$

We can apply \eqref{eq:wh-Bernstein} or, equivalently, \eqref{eq:MP1general}, to find the coefficients of $N_{d,m}(x)$ in the Bernstein basis $x^{j}(1-x)^{K-j}$. Indeed, \eqref{eq:wh-Bernstein} states that 
\begin{equation}\label{eq:NBernstein1}
N_{d,m}(x)=(1-x)^{K}w\bigg(m+d;F_K\bigg\vert\:\frac{x}{x-1}\bigg)=\sum\limits_{j=0}^{K}\frac{(-K)_j(m+d)_j}{(-1)^{j}(2)_jj!}Q_{L}(j)x^j(1-x)^{K-j}.    
\end{equation}
The characteristic polynomial  $Q_{L}(t)=[T_L(m+1;2)H_{L}](t)$, where $H_{L}$ is defined  in \eqref{eq:HM}, has degree $L=(d-2)(m-1)$.
    
Expansion \eqref{eq:NBernstein1} is not unique in the following sense. Writing the connection formula \eqref{eq:whF-whR-connection}, in view of \eqref{eq:dNarayanaFEuler}, yields:
\begin{equation}\label{eq:connectionforNarayana}
(1-y)^{K}N_{d,m}\bigg(\frac{y}{y-1}\bigg)
=\hat{w}\bigg(\!\!-(d-1)m;\frac{(2+\cdot)_{m-1}}{(2)_{m-1}}T_{L}(m+d;2)H_{L}\bigg\vert\,y\bigg),
\end{equation}
which, after using definition \eqref{eq:fa-Euler} and $x=y/(y-1)$, amounts to
$$
N_{d,m}(x)=
\sum_{j=0}^{M}\frac{(-M)_{j}(m+1)_{j}}{(-1)^{j}(2)_{j}j!}\widetilde{Q}_{L}(j)x^{j}(1-x)^{M-j},
$$
where $\widetilde{Q}_{L}(t)=[T_{L}(m+d;2)H_{L}](t)$.
Choosing different parameters as $\delta$, $\epsilon$, $\rho$ when applying the Miller-Paris transformations we can get many different representations of this type. For example, if  $R_{L-1}(t)=\big[T_{(d-2)m}(m+1;2)P_{(d-2)m)}\big](t)$ with
\begin{equation}\label{eq:Pd-2m}
P_{(d-2)m}(x)=\frac{(3+x)_{m}\cdots(d+x)_{m}}{(3)_{m}\cdots(d)_m}    
\end{equation}
of degree $(d-2)m$, then 
\begin{equation}\label{eq:NBern1}
N_{d,m}(x)=(1-x)^{M-1}{}_{L+1}F_{L}\!\left(\!\!\begin{array}{c}1-M,m+2,\zetta+1\\2,\zetta\end{array}\vline\:\frac{x}{x-1}\right),
\end{equation}
where $\zetta=(\zeta_1,\zeta_2,\ldots,\zeta_{L-1})$ are the roots of $R_{L-1}(-t)$.  Notation $R_{L-1}$ is justified by the fact that the true degree of $R_{L-1}$ is $L-1$. This can be  seen as follows. We can rewrite  formula \cite[(10)]{Karp2025}  as 
	\begin{equation*}
		\frac{T_{\omega}(\epsilon;\rho)(\nuu+\cdot)_{\omm}}{(\nuu)_{\omm}}=\frac{1}{(\nuu)_{\omm}(\rho-\epsilon-\omega)_{\omega}}\sum\limits_{k=0}^{\omega}\frac{(\epsilon)_k(1-\rho+\epsilon)_{k}(\rho-\epsilon-\omega+t)_{\omega-k}}{(-1)^kk!}\sum\limits_{j=0}^{k}\frac{(-k)_{j}}{j!}(\nuu-\epsilon-j)_{\omm}.
	\end{equation*}
 If we apply it to the definition of $R_{L-1}$ with $\rho=2$, $\epsilon=m+1$,  $\nuu=(3,\ldots,d)$, $\omm=(m,\ldots,m)$, $\omega=(d-2)m$, we will have 
	\begin{multline*}
		R_{L-1}(t)=\frac{1}{(3)_{m}\cdots(d)_{m}(1-(d-1)m)_{(d-2)m}}\sum\limits_{k=0}^{(d-2)m}\frac{(m)_{k}(m+1)_k}{(-1)^kk!}(1-(d-1)m+t)_{(d-2)m-k}
		\\
		\times\sum\limits_{j=0}^{k}\frac{(-k)_{j}}{j!}(2-m-j)_{m}\cdots(d-1-m-j)_{m},
	\end{multline*}
As, clearly,
	$$
	(2-m-j)_{m}(2-m-j)_{m}\cdots(d-1-m-j)_{m}=0~\text{for}~j=0,1,\ldots,d-2,
	$$
	the inner sum vanishes for $k=0,1,\ldots,d-2$, and the outer sum starts with $k=d-1$ proving that the degree of $R_{L-1}$ is indeed $(d-2)m-(d-1)=L-1$.

\subsection{Zeros of \texorpdfstring{$N_{d,m}(x)$}{}}
We intend to apply Theorem~\ref{th:negative_zeros} to the polynomial $N_{d,m}(x)=\hat{w}(m+d;F_K\vert\,x)$, see \eqref{eq:dNarayanaFEuler}.  The negated zeros of the polynomial $F_K$ defined in \eqref{eq:FM} are 
$$
\{\alpha_1,\ldots,\alpha_K\}=\{2,\ldots,m,3,\ldots,m+1,\ldots,d,\ldots,m+d-2\}.
$$
As $a=m+d$, we have $a>\alpha_{j}$, $j=1,\ldots,K$.  Hence, by Corollary~\ref{cr:large-a} all zeros of $N_{d,m}$ are real and negative.

\subsection{Multidimensional Catalan numbers \texorpdfstring{$N_{d,m}(1)$}{}}
	The value $N_{d,m}(1)$  can be found immediately from the asymptotic formula \cite[(3.1)]{CKP2021}
	$$
	{}_{d}F_{d-1}\left.\!\left(\begin{matrix}\a\\ \b\end{matrix}\:\right\vert x\right)
	=\frac{\Gamma(\b)}{\Gamma(\a)}\Gamma\Big(\sum a_i-\sum b_j\Big)(1-x)^{\sum b_j-\sum a_i}\left(1+o(1)\right)~\text{as}~x\to1
	$$
	valid if $\sum b_j<\sum a_i$.  Application of this formula yields as $x\to1$:
	$$
	{_dF_{d-1}}\!\left(\begin{array}{c}m+1,m+2,\ldots,m+d\\2,3,\ldots,d\end{array}\vline\:x\right)=(1-x)^{-md-1}\frac{1!2!\cdots (d-1)!(md)!}{m!(m+1)!\cdots(m+d-1)!}\left(1+o(1)\right),
	$$
	so that  multidimensional Catalan numbers are given by
	$$
	N_{d,m}(1)=\frac{1!2!\cdots (d-1)!(md)!}{m!(m+1)!\cdots(m+d-1)!}
	$$
    in accord with \eqref{eq:d-Catalan}.
	
	\bigskip
	\bigskip
	
	\subsection{Palindromic  property (or self-reciprocity)}
	To prove the palindromic property it is convenient to use the expression for the polynomial $N_{d,m}(x)$   obtained directly from \eqref{eq:CiglerF} in view of \eqref{eq:wh-defined}: 
	\begin{multline}\label{eq:Amd2}
		N_{d,m}(x)=(1-x)^{md+1}{_dF_{d-1}}\!\left(\begin{array}{c}m+1,m+2,\ldots,m+d\\2,3,\ldots,d\end{array}\vline\:x\right)
		\\
		=\sum\limits_{n=0}^{(d-1)(m-1)}\!\!\!x^n\frac{(-md-1)_n}{n!}{_{d+1}F_{d}}\!\left(\!\!\begin{array}{c}-n,m+1,m+2,\ldots,m+d\\md+2-n,2,3,\ldots,d\!\!\end{array}\right).
	\end{multline} 
	The upper summation limit is $(d-1)(m-1)$ due to the  degree of $N_{d,m}(x)$. Note that 
\eqref{eq:Amd2} is easily seen to be equivalent to Sulanke's formula \cite[Proposition~1]{Sulanke2004}. 
	Using \eqref{eq:Amd2} the required palindromic property $[x^n]N_{d,m}(x)=[x^{(d-1)(m-1)-n}]N_{d,m}(x)$ takes the form
	\begin{multline}\label{eq:panindrom}
		\frac{(-md-1)_n}{n!}{_{d+1}F_{d}}\!\left(\!\!\begin{array}{c}-n,m+1,m+2,\ldots,m+d\\md+2-n,2,3,\ldots,d\!\!\end{array}\right)
		\\
		=\frac{(-md-1)_{(d-1)(m-1)-n}}{((d-1)(m-1)-n)!}{_{d+1}F_{d}}\!\left(\!\!\begin{array}{c}n-(d-1)(m-1),m+1,m+2,\ldots,m+d\\1+m+d+n,2,3,\ldots,d\!\!\end{array}\right).
	\end{multline}
	This identity is an instance of Gasper's identity \cite[Theorem~2.3]{KPITSF2018}
	\begin{equation}\label{eq:Kar4}
		\frac{1}{\Gamma(1-a)}{}_{r+2}F_{r+1}\left.\!\!\left(\begin{matrix}a,b,\f+\m\\c,\f\end{matrix}\right.\right)
		=
		\frac{\Gamma(c)(\f-b)_\m}{\Gamma(1+b-a)\Gamma(c-b)(\f)_\m}
		{}_{r+2}F_{r+1}\left.\!\!\left(\begin{matrix}b,1-c+b,1-\f+b\\1+b-a,1-\f-\m+b\end{matrix}\right.\right)
	\end{equation}
	valid for $\Re(c-a-b-\sum m_i)>0$, $-c\notin\N\cup\{0\}$ and $-f_i\notin\N\cup\{0\}$ with identification:
	$$
	a=-n,~~~b=m+d,~~~c=md+2-n,~~~f=(2,\ldots,d),~~~\m=(m-1,\ldots,m-1).
	$$

\medskip 	
\textbf{Remark.} The claim that $N_{d,m}(x)$ is palindromic established above is equivalent to the equality $N_{d,m}(x)=x^{K}N_{d,m}(1/x)$. On comparing coefficients in \eqref{eq:Narayana_explicit} 
this amounts to
$$
\binom{M}{j}\binom{K+1}{j+1}\hat{Q}_{L}(j)=
\binom{M}{K-j}\binom{K+1}{K-j+1}\hat{Q}_{L}(K-j).
$$
Simplifying binomials this is equivalent to 
$$
(j+2)_{d-2}\hat{Q}_{L}(j)=
(K-j+2)_{d-2}\hat{Q}_{L}(K-j)
$$
for $j=0,\ldots,K$. Both sides of the above identity are polynomials of degree $(d-2)m$ in the variable $j$.  For $m>d-2$ the number of values of  $j$, where they agree,  exceeds their degrees, so that both polynomials must coincide identically. This amounts to various curious  hypergeometric identities of the form
$$
(t+2)_{d-2}\hat{Q}_{L}(t)=
(K-t+2)_{d-2}\hat{Q}_{L}(K-t)
$$
where $\hat{Q}_{L}(t)=[\hat{T}_{L}(m+1,m+d;2)H_{L}(\cdot)](t)$ and $H_{L}$ defined in \eqref{eq:HM}.  We can use various explicit forms for  the second Miller-Paris operator $\hat{T}_{L}(m+1,m+d;2)$ by choosing any of the expressions \cite[(19), (20), (24), (25)]{Karp2025}, while the choices on the right and left hand sides may differ. One particular choice is given in  \eqref{eq:hatQN}.

	\section{Jacobi-Pi\~neiro polynomials}
	
	\subsection{Zeros of Jacobi-Pi\~neiro polynomials}
	
	In view of definition \eqref{eq:JacobiPineiro} we can identify the Jacobi-Pi\~neiro polynomial with the generalized $f$-Eulerian polynomial $\hat{w}$ as defined in \eqref{eq:fa-Euler} or \eqref{eq:fa-Euler1} (recall that $(\boldsymbol{\alpha}+1)_{\n}=(\alpha_1+1)_{n_1}\cdots(\alpha_r+1)_{n_r}$):
	$$
	P_{\n}^{(\boldsymbol{\alpha},\beta)}(x)=\frac{(-1)^n(\boldsymbol{\alpha}+1)_{\n}}{(n+\boldsymbol{\alpha}+\beta+1)_{\n}}\hat{w}(-n-\beta;F_n\vert\:x),
	$$
	where $\n=(n_1,\ldots,n_r)$, $n=n_1+\cdots+n_r$, and
$$
F_n(t)=\frac{(\boldsymbol{\alpha}+1+t)_{\n}}{(\boldsymbol{\alpha}+1)_{\n}}.
$$
We will apply Theorem~\ref{th:zeros-in-0-1} to determine the zero location of $P_{\n}^{(\boldsymbol{\alpha},\beta)}$. Indeed, once we assume that $\alpha_j>-1$ for $j=1,\ldots,r$, we immediately conclude that all zeros of $F_n(t)$ are negative. Hence, according to Theorem~\ref{th:zeros-in-0-1} if $a+n=-n-\beta+n=-\beta<1$, i.e. if $\beta>-1$, then all zeros of $P_{\n}^{(\boldsymbol{\alpha},\beta)}$ lie in $(0,1)$. If  $\alpha_j>-1$ for $j=1,\ldots,r$, while $-2<\beta\le-1$, all zeros but one lie in $(0,1)$ and the remaining zero lies in $[1,\infty)$. Moreover, by an application of Lemma~\ref{lm:Talpha-negative_s} it is easy to  see that the zeros $P_{\n}^{(\boldsymbol{\alpha},\beta)}(x)$ and $P_{\m}^{(\boldsymbol{\eta},\beta)}(x)$ interlace if $n-m=1$ (here $m=m_1+\cdots+m_k$) and 
	$$
	(\eta_1,\eta_1+1,\ldots,\eta_1+m_1-1,\ldots,\eta_{k},\ldots,\eta_{k}+m_k-1)\!\subset\!(\alpha_1,\alpha_1+1,\ldots,\alpha_1+n_1-1,\ldots,\alpha_{r},\ldots,\alpha_{r}+n_r-1).
	$$

\subsection{Explicit expressions}
    
	Explicit expansion in monomial basis for $P_{\n}^{(\boldsymbol{\alpha},\beta)}(x)$ can be constructed using \eqref{eq:wh-formula} from Proposition~\ref{pr:fEuler-monomial}:
	\begin{equation*}
		P_{\n}^{(\boldsymbol{\alpha},\beta)}(x)=\frac{(-1)^n(\alpha_1+1)_{n_1}}{(n+\boldsymbol{\alpha}+\beta+1)_{\n}}\sum\limits_{k=0}^{n}x^k
		\frac{(-n)_{k}(\alpha_1+n_1+\beta+1)_{k}}{(\alpha_1+1)_{k}k!}\hat{Q}_{\tilde{n}}(k),
	\end{equation*}
	where $\tilde{n}=n-n_1$ and 
	$$
	\hat{Q}_{\tilde{n}}(t)=\hat{T}_{\tilde{n}}(-n-\beta,\alpha_1+n_1+1;\alpha_1+1)(\boldsymbol{\alpha}_{[1]}+t+1)_{\n_{[1]}}.
	$$
	Any of the expressions 
    \cite[(19), (20), (24), (25)]{Karp2025} can be used for computing the second Miller-Paris operator $\hat{T}_{\tilde{n}}$.
Of course, the above representation is not unique in the sense that we can exchange the roles of $\alpha_1$ and any other $\alpha_l$, $l=2,\ldots,r$.  
    
Furthermore, formula \eqref{eq:wh-Bernstein} gives an expansion of $P_{\n}^{(\boldsymbol{\alpha},\beta)}(x)$  in the Bernstein basis:
	$$
	P_{\n}^{(\boldsymbol{\alpha},\beta)}(x)=\frac{(-1)^n(\alpha_1+1)_{n_1}}{(n+\boldsymbol{\alpha}+\beta+1)_{\n}}\sum\limits_{k=0}^{n}(-1)^k\frac{(-n)_{k}(-n-\beta)_k}{(\alpha_1+1)_{k}k!}
	Q_{\tilde{n}}(k)x^k(1-x)^{n-k}, 	
	$$
	where 
	$$
	Q_{\tilde{n}}(t)=T_{\tilde{n}}(\alpha_1+n_1+1;\alpha_1+1)(\boldsymbol{\alpha}_{[1]}+t+1)_{\n_{[1]}}.
	$$
	Any of the expressions \cite[(8)--(13)]{Karp2025} can be used for computing the first Miller-Paris operator $T_{\tilde{n}}$.
	
	Next, we explore the consequences of the relation \eqref{eq:whF-whR-connection}. Its application to $\hat{w}(-n-\beta;F_{n}|x)$ (where $x=y/(y-1)$) in view of \eqref{eq:fa-Euler1} and \eqref{eq:Fm-defined} yields 
	$$
	P_{\n}^{(\boldsymbol{\alpha},\beta)}\Big(\frac{y}{y-1}\Big)=\frac{(-1)^n(\alpha_1+1)_{n_1}(1-y)^{\alpha_1+n_1+\beta+1}}{(n+\boldsymbol{\alpha}+\beta+1)_{\n}}F\left(\!\!\!\begin{array}{c}
		\alpha_1+n_1+\beta+1, \alpha_1+n_1+1 \\
		\alpha_1+1
	\end{array} \bigg\vert\hat{R}_{\tilde{n}}\bigg\vert\: y\right),
	$$
	where $\tilde{n}=n-n_1$ and 
	$$
	\hat{R}_{\tilde{n}}(t)=T_{\tilde{n}}(-n-\beta;\alpha_1+1)(\boldsymbol{\alpha}_{[1]}+1+\cdot)_{\n_{[1]}}.
	$$
	As $y\in(-\infty,0)$ is equivalent to $y/(y-1)\in(0,1)$, we see that if all (or some) zeros of  the hypergeometric function of the right hand side lie in $(-\infty,0)$, then all (or some) zeros of  $P_{\n}^{(\boldsymbol{\alpha},\beta)}(x)$ lie in $(0,1)$. Writing $\boldsymbol{\eta}$ for the vector of zeros of $\hat{R}_{\tilde{n}}(-t)$ we can rewrite the above formula as
\begin{multline*}
P_{\n}^{(\boldsymbol{\alpha},\beta)}\Big(\frac{y}{y-1}\Big)=\frac{(-1)^n(\alpha_1+1)_{n_1}(1-y)^{\alpha_1+n_1+\beta+1}}{(n+\boldsymbol{\alpha}+\beta+1)_{\n}}
\\
\times{}_{\tilde{n}+2}F_{\tilde{n}+1}\left(\!\!\!\begin{array}{c}
		\alpha_1+n_1+1+\beta, \alpha_1+n_1+1,\eta_1+1,\ldots,\eta_{\tilde{n}}+1 \\
		\alpha_1+1,\eta_1,\ldots,\eta_{\tilde{n}}
	\end{array} \bigg\vert\: y\right).
\end{multline*}
Whenever Theorem~\ref{th:negative_zeros} is applicable to this function, we can then conclude that the zeros of  $P_{\n}^{(\boldsymbol{\alpha},\beta)}(x)$ lie in $(0,1)$.

One possible form of $\hat{R}_{\tilde{n}}$ according to  \cite[(9)]{Karp2025} is
	$$
	\hat{R}_{\tilde{n}}(t)\!=\!\frac{(\boldsymbol{\alpha}_{[1]}+1)_{\n_{[1]}}}{(\alpha_1+1+n_1+\beta)_{\tilde{n}}}\sum\limits_{k=0}^{\tilde{n}}\frac{(-n-\beta)_k(-t)_{k}(\alpha_1+1+n_1+\beta+t)_{\tilde{n}-k}}{(-1)^kk!}{}_{r}F_{r-1}\!\left(\begin{matrix}-k,\boldsymbol{\alpha}_{[1]}+1+\n_{[1]}\\\boldsymbol{\alpha}_{[1]}+1\end{matrix}\right).
	$$
	In the case $r=2$ of two weights,  $\boldsymbol{\alpha}=(\alpha_1,\alpha_2)$ and the hypergeometric function in the summand becomes 
	$$
	{}_{2}F_{1}\!\left(\begin{matrix}-k,\alpha_2+1+n_2\\\alpha_2+1\end{matrix}\right)=\frac{(-n_2)_k}{(\alpha_2+1)_k}
	$$
	by Chu-Vandermonde's identity. Hence, in this case, in view of $(\gamma)_{n-k}=(-1)^k(\gamma)_{n}/(1-\gamma-n)_{k}$, we have
	$$
	\hat{R}_{n_2}(t)=\frac{(\alpha_2+1)_{n_2}(\alpha_1+1+n_1+\beta+t)_{n_2}}{(\alpha_1+1+n_1+\beta)_{n_2}}{}_{3}F_{2}\!\left(\begin{matrix}-n_2,-t,-n-\beta\\-\alpha_1-n-\beta-t,\alpha_2+1\end{matrix}\right).
	$$
	Further by an application of Sheppard's transformation \cite[Appendix (II)]{RJRJR1992} this can be transformed into
	$$
	\hat{R}_{n_2}(t)=\frac{(-1)^{n_2}(\alpha_2+1)_{n_2}(\alpha_2-\alpha_1)_{n_2}}{(\alpha_1+1+n_1+\beta)_{n_2}}{}_{3}F_{2}\!\left(\begin{matrix}-n_2,\alpha_2+1+t,n+\alpha_2+\beta+1\\\alpha_2+1,\alpha_2+1-\alpha_1\end{matrix}\right).
	$$
	If $\alpha_1=\alpha_2+N+1$, where $N\ge{n_2}$ is integer, the ${}_3F_2$ polynomial on the right hand side can be identified with Hahn's discrete orthogonal polynomial $Q_{n_2}(-t-\alpha_2-1,\alpha_2,\beta+n_1;N)$ \cite[(18.20.5)]{NIST}.  These polynomials are orthogonal with respect to a positive discrete weight under conditions $\alpha_2,\beta+n_1>-1$ or $\alpha_2,\beta+n_1<-N$. It follows from the general theory of orthogonal polynomials that under these conditions all zeros of $\hat{R}_{n_2}(t)$ are real and lie in the interval $-\alpha_1<t<-\alpha_2-1$.

	We can also work in the opposite direction. Start with the function 
\begin{equation}\label{eq:JPcompanion}
	F\left(\!\!\!\begin{array}{c}
		\alpha_1+n_1+\beta+1, \alpha_1+n_1+1 \\
		\alpha_1+1
	\end{array} \bigg\vert P_{\tilde{n}}\bigg\vert\: y\right),
	\end{equation}
	where $P_{\tilde{n}}$ is a polynomial of degree $\tilde{n}$ with known factorization 
	$$
	P_{\tilde{n}}(t)=(t+\eta_1)\cdots(t+\eta_{\tilde{n}})
	$$  
satisfying the conditions of Theorem~\ref{th:negative_zeros}  with $a=\alpha_1+n_1+\beta+1$.    Applying transformation \eqref{eq:MP1general} with $\delta=\alpha_1+n_1+1$, $\epsilon=\alpha_1+n_1+1$ and $\rho=\alpha_1+1$  we will get:
	$$
	(1-y)^{-\alpha_1-n_1-a}
	F\left(\!\!\!\begin{array}{c} -\beta-n,\alpha_1+n_1+1
\\
		\alpha_1+1
	\end{array} \bigg\vert F_{\tilde{n}}\bigg\vert\: \frac{y}{y-1}\right),
	$$
	where 
	$$
	F_{\tilde{n}}(t)=\big[T_{\tilde{n}}(\alpha_1+n_1+\beta+1;\alpha_1+1)P_{\tilde{n}}\big](t)=A(t+\gamma_1)\cdots(t+\gamma_{\tilde{n}}).
	$$
	Define $\boldsymbol{\gamma}=(\alpha_1+1,\ldots,\alpha_1+n_1-1,\gamma_1,\ldots,\gamma_{\tilde{n}})$.
Then we obtain 	
$$
(1-x)^{\beta}F\left(\!\!\!\begin{array}{c} -\beta-n,\alpha_1+n_1+1
	\\
	\alpha_1+1
\end{array} \bigg\vert F_{\tilde{n}}\bigg\vert\: x\right)=
\frac{ \left(n+\boldsymbol{\gamma}+\beta+1\right)_{\n}}{(-1)^{n}\left(\boldsymbol{\gamma}+1\right)_{\n}}
P_{\n}^{(\boldsymbol{\gamma},\beta)}(x),
	$$
and this implies that all zeros of $P_{\n}^{(\boldsymbol{\gamma},\beta)}(x)$ lie in $(0,1)$ once the function \eqref{eq:JPcompanion} satisfies the hypothesis of Theorem~\ref{th:negative_zeros}.	

\subsection{Connection with \texorpdfstring{$d$}{d}-Narayana polynomials}

There are many ways to express $d$-Narayana polynomial as a particular Jacobi-Pi\~neiro polynomial. According to definition \eqref{eq:JacobiPineiro} Jacobi-Pi\~neiro polynomials are essentially the same as generalized $f$-Eulerian polynomials \eqref{eq:fa-Euler}. This implies that every hypergeometric expression for $d$-Narayana polynomial can be viewed as  a special case of the Jacobi-Pi\~neiro polynomial. For instance, comparing \eqref{eq:CiglerF} with \eqref{eq:JacobiPineiro} we can write
$$
N_{d,m}(x)=
\frac{ \left(n+\boldsymbol{\alpha}+\beta+1\right)_{\n}}{(-1)^{n}\left(\boldsymbol{\alpha}+1\right)_{\n}}P_{\n}^{(\boldsymbol{\alpha}, \beta)}(x),
$$
where $\boldsymbol{\alpha}=(1,2,\ldots,d-1)$,
$\n=(m,\ldots,m)$, $\beta=-md-1$.  This formula, however, is not of particular interest, since parameter the  $\beta$ depends on $m$.   Note that $\beta$ is equal to the  parametric excess (i.e. the sum of the bottom parameters minus the sum of the top parameters) of the ${}_{r+1}F_{r}$ function on the right hand side of \eqref{eq:JacobiPineiro}.  
We can get $\beta$ independent of $m$ by employing various representations for  $N_{d,m}(y/(y-1))$.  For instance, by the first equality in \eqref{eq:NBernstein1}, \eqref{eq:MP1poly_reduced}, \eqref{eq:JacobiPineiro} and \eqref{eq:HM}, we obtain
$$
N_{d,m}\bigg(\frac{y}{y-1}\bigg)=
\frac{1}{(1-y)^{K}}F\!\left(\begin{matrix}-K,m+d\\2\end{matrix}\,\bigg\vert\,Q_{L}\,\bigg\vert\, y\right)=\frac{(-1)^{n}\left(n+\boldsymbol{\alpha}+\beta+1\right)_{\n}}{(1-y)^{M}\left(\boldsymbol{\alpha}+1\right)_{\n}}P_{\n}^{(\boldsymbol{\alpha}, \beta)}(y)
$$	
with $M=(d-1)m$, $K=(d-1)(m-1)$, $L=(d-2)(m-1)$, $\beta=1-d$,
$$
Q_{L}(t)=\big[T_{L}(m+1; 2)H_L\big](t)=\frac{\big[T_{L}(m+1; 2)(3+\cdot)_{m-1}\cdots(d+\cdot)_{m-1}\big](t)}{(3)_{m-1}\cdots(d)_{m-1}}
$$ 
and $\alpha_1=1$, $n_1=m+d-2$, $\alpha_2=\eta_1-1,\ldots,\alpha_{L+1}=\eta_L-1$, $n_2=\cdots=n_{L+1}=1$, where $\etta$ are the roots of  $Q_{L}(-t)$. The number of parameters $\alpha_j$ (i.e. the number of orthogonality weights for $P_{\n}^{(\boldsymbol{\alpha}, \beta)}(y)$) unfortunately depends on $d$ and $m$ except for the case $d=2$ established in \cite{KMFS}.
    
Alternatively, combining formula \eqref{eq:connectionforNarayana} with the definitions \eqref{eq:fa-Euler1} of $f$-Eulerian and  \eqref{eq:JacobiPineiro}  of Jacobi-Pi\~neiro polynomials yields: 
$$
N_{d,m}\bigg(\frac{y}{y-1}\bigg)=\frac{(-1)^{n}\left(K+\boldsymbol{\alpha}+\beta+1\right)_{\n}}{(1-y)^{K}\left(\boldsymbol{\alpha}+1\right)_{\n}}P_{\n}^{(\boldsymbol{\alpha}, \beta)}(y),
$$
where $\beta=d-1$,  $\alpha_1=1$, $n_1=m-1$, $n_2=\cdots=n_{L+1}=1$  and $\alpha_j=\eta_{j-1}-1$, $j=2,\ldots,L+1$, where $\etta$ is the vector of $L$ roots of $\big[T_L(m+d;2)H_L\big](-t)$ with $H_L$ from \eqref{eq:HM}.

If, instead of  \eqref{eq:connectionforNarayana} we use \eqref{eq:NBern1}, the result becomes
$$
N_{d,m}\bigg(\frac{y}{y-1}\bigg)=\frac{(-1)^{n}\left(K+\boldsymbol{\alpha}+\beta+1\right)_{\n}}{(1-y)^{K}\left(\boldsymbol{\alpha}+1\right)_{\n}}P_{\n}^{(\boldsymbol{\alpha}, \beta)}(y),
$$
where $\beta=d-2$,  $\alpha_1=1$, $n_1=m$, $n_2=\cdots=n_{L}=1$  and $\alpha_j=\zeta_{j-1}-1$, $j=2,\ldots,L$, where $\zetta$ is the vector of negated $L-1$ roots of $R_{L-1}=T_{(d-2)m}(m+1;2)P_{(d-2)m}$ with $P_{(d-2)m}$ from \eqref{eq:Pd-2m}. Note, that unlike previous two examples here we obtain the polynomial with $L$ (and not $L+1$) weights.

As all zeros of $N_{d,m}(x)$ lie in $(-\infty,0)$ it follows that all zeros of $P_{\n}^{(\boldsymbol{\alpha},\beta)}(y)$ in all the above examples lie in $(0,1)$. At the same time, one can verify  numerically that some of the parameters $\boldsymbol{\alpha}$ are complex for sufficiently large $m$ (for $d=5$ and $m\ge5$ in the ultimate example), so that the location of zeros does not follow from orthogonality.

\end{document}